 \documentclass[draft]{article}

\usepackage{amsmath,amsfonts,amsthm,amssymb,amscd,cancel,color}
\usepackage{enumitem}
\usepackage{verbatim}
\usepackage[dvipdfmx]{graphicx}
\usepackage{ulem,color}

\renewcommand{\Re}{\mathop{\rm Re}\nolimits}
\renewcommand{\Im}{\mathop{\rm Im}\nolimits}

\theoremstyle{plain}
\newtheorem{theorem}{Theorem}[section]
\newtheorem{lemma}[theorem]{Lemma}
\newtheorem{proposition}[theorem]{Proposition}
\newtheorem{corollary}[theorem]{Corollary}
\theoremstyle{definition}
\newtheorem{definition}[theorem]{Definition}
\theoremstyle{remark}
\newtheorem{remark}[theorem]{Remark}

\newtheorem{assumption}[theorem]{Assumption}
\newtheorem{claim}[theorem]{Claim}

\newcommand{\R}{{\mathbb R}}

\newcommand{\Z}{{\mathbb Z}}

\newcommand{\N}{{\mathbb N}}

\def\im{{\rm i}}

\newcommand{\C}{\mathbb{C}}

\newcommand{\fx}{\mathfrak{X}}
\newcommand{\fxr}{\mathfrak{X}_{\sharp}}

\newcommand{\Or}{\Omega_{\sharp}}

\def\({\left(}
\def\){\right)}
\def\<{\left\langle}
\def\>{\right\rangle}

\numberwithin{equation}{section}

\begin{document}

\title{   Coordinates  at small energy and refined profiles for the Nonlinear Schr\"odinger Equation}

\author{Scipio Cuccagna, Masaya Maeda}
\maketitle

\begin{abstract} In this paper we give a new and simplified proof of the theorem on selection of  standing waves
for small energy solutions of the nonlinear Schr\"odinger equations (NLS) that we gave in  \cite{CM15APDE}. We consider a  NLS with a
Schr\"odinger operator with several eigenvalues, with corresponding families of small standing waves, and we show that
any  small energy solution  converges   to the orbit of a time periodic solution  plus a scattering term. The novel idea
is to consider the "refined   profile",   a  quasi--periodic function in time which almost solves   the  NLS and   encodes
the discrete modes of a solution. The refined   profile, obtained by elementary means,  gives us directly an optimal coordinate system,    avoiding  the normal form
arguments in  \cite{CM15APDE}, giving us also a better understanding of the Fermi Golden Rule.
\end{abstract}

\section{Introduction}

In this paper, we consider the following nonlinear Schr\"odinger equation (NLS):
\begin{equation}\label{nls}
\im \partial_t u = Hu + g(|u|^2)u,\quad (t,x)\in \R^{1+3}.
\end{equation}
Here $H:=-\Delta +V$ is  a  Schr\"odinger operator with $V\in \mathcal S(\R^3,\R)$ (Schwartz function). For the nonlinear term we require $g \in C^\infty (\R,\R)$ with $g(0)=0$ and the growth condition:
\begin{align}\label{eq:ggrowth}
\forall n\in\N\cup\{0\},\ \exists C_n>0,\ |g^{(n)}(s)|\leq  C_n \<s\>^{2-n}\text{ where } \<s\>:=(1+|s|^2)^{1/2}.
\end{align}
We consider the Cauchy problem of NLS \eqref{nls} with the initial condition $u(0)=u_0 \in H^1(\R^3,\C)$.
It is well known that NLS \eqref{nls} is locally well-posed (LWP) in $H^1$, see e.g.\  \cite{CazSemi,LPBook
}.

The aim of this paper is to revisit the study of asymptotic behavior of small (in $H^1$) solutions when the Schr\"odinger operator $H$ has several simple eigenvalues.
In such situation, it have been proved that solutions decouple into a soliton and dispersive wave \cite{SW04RMP,TY02ATMP,CM15APDE}.

To state our main result precisely, we introduce some notation  and several assumptions.
The following two assumptions for the Schr\"odinger operator $H$ hold
  for generic $V$.

\begin{assumption}\label{ass:nonres}
$0$ is neither an eigenvalue nor a resonance of $H$.
\end{assumption}

\begin{assumption}\label{ass:linearInd}
There exists $N\geq 2$ s.t.\
$$\sigma_d(H)=\{\omega_j\ |\ j=1,\cdots,N\},\text{ with }\omega_1<\cdots<\omega_N<0,$$
where $\sigma_d(H)$ is the set of discrete spectrum of $H$.
Moreover, we assume all $\omega_j$ are simple and
\begin{align}\label{eq:linind}
\forall \mathbf{m} \in \Z^N\setminus\{0\},\ \mathbf{m} \cdot  \boldsymbol{\omega} \neq 0,
\end{align}
where $\boldsymbol{\omega}:=(\omega_1,\cdots, \omega_N)$.
We set $\phi_j$ to be the eigenfunction of $H$ associated to the eigenvalue $\omega_j$ satisfying $\|\phi_j\|_{L^2}=1$.
We also set $\boldsymbol{\phi}=(\phi_1,\cdots,\phi_N)$.
\end{assumption}

\begin{remark}
	The cases $N=0,1$ are easier and are not treated it in this paper.
	Unfortunately, Assumption \eqref{ass:linearInd} excludes      radial potentials $V(r)$, for $r=|x|$, where in general we should expect
eigenvalues with   multiplicity higher than one. In fact the symmetries imply that  each eigenspace $\ker (H-\omega _j)$ is spanned by functions
which in spherical coordinates are separated and are of form $ \frac{1}{r}u_{j,l}(r)  e^{\im m \theta }  P_l^{m}(\cos (\varphi)) $  for appropriate $l\in \N \cup \{ 0\}$ with $P_l^{m}$
Legendre polinomials, and $m$ taking all values between $-l$ and $l$, so that, if $l\ge 1$, the multiplicity is at least $2l+1$. See p. 778 \cite{Cohen-Tannoudji}.

\end{remark}

As it is well known, $\phi_j$'s are smooth and decays exponentially.
For $s\geq 0, \gamma\geq 0$, we set
\begin{align*}
H^s_\gamma:=\{u\in H^s\ |\ \|u\|_{H^s_\gamma}:=\|\cosh(\gamma x)u\|_{H^s}<\infty\}.
\end{align*}
The following is well known.
\begin{proposition}\label{prop:gam}
There exists $\gamma_0>0$ s.t.\ for all $1\leq j\leq N$, we have $\phi_j\in \cap_{s\geq 0}H^s_{\gamma_0}$.
\end{proposition}


Using $\gamma_0>0$, we set
\begin{align*}
\Sigma^s:= H^s_{\gamma_0} \ \text{if }s\geq 0,\  \Sigma^{s}:=(H^{-s}_{\gamma_0})^* \ \text{ if } s<0, \ \Sigma^{0-}:=(\Sigma^0)^*\text{ and }\Sigma^\infty:=\cap_{s\geq 0}\Sigma^s.
\end{align*}
We will not consider any topology in $\Sigma^\infty$ and we will only consider it as a set.

In order to introduce the   notion of  refined profile, we need the following
  combinatorial set up.

\noindent We start the following standard basis of $\R ^N$, which we view as ``non--resonant" indices:
\begin{align}\label{eq:defnr0}
\mathbf{NR}_0:=\{\mathbf e_j \ |\ j=1,\cdots,N\},\ \mathbf e_j:=(\delta_{1j},\cdots, \delta_{Nj})\in \Z^N, \text{ $\delta_{ij}$   the Kronecker delta.}
\end{align}
More generally, the
  sets of resonant and non--resonant indices $\mathbf{R}$, $\mathbf{NR}$, are
\begin{align}\label{eq:defrnr}
\mathbf{R}:=\{\mathbf{m}\in \Z^N\ |\ \sum\mathbf{m}=1,\ \boldsymbol{\omega} \cdot \mathbf{m}>0\},\quad
\mathbf{NR}:=\{\mathbf{m}\in \Z^N\ |\ \sum\mathbf{m}=1,\ \boldsymbol{\omega} \cdot \mathbf{m}<0\},
\end{align}
where $\sum \mathbf{m}:=\sum_{j=1}^N m_j$ for $\mathbf{m}=(m_1,\cdots,m_N) \in \Z^N$.

\noindent From Assumption \ref{ass:linearInd}  it is clear that $\{\mathbf{m}\in \Z^N\ |\ \sum \mathbf{m}=1\}=\mathbf{R}\cup \mathbf{NR}$ and $\mathbf{NR}_0\subset \mathbf{NR}$.
For $\mathbf{m} =(m_1,\cdots, m_N)\in \Z^N$, we define
\begin{align}\label{eq:absmdef}
|\mathbf{m}|:=(|m_1|,\cdots,|m_N|)\in \Z^N,\ \|\mathbf{m}\|:=\sum |\mathbf{m}|=\sum_{j=1}^N |m_j|,
\end{align}
and introduce partial orders $\preceq$ and $\prec$ by
\begin{align}
\mathbf{m}\preceq \mathbf{n}\ \Leftrightarrow_{\mathrm{def}} \forall j\in \{1,\cdots,N\},\ m_j\leq n_j, \label{eq:defpreceq}\quad \text{and}\quad
\mathbf{m}\prec \mathbf{n} \ \Leftrightarrow_{\mathrm{def}} \mathbf{m}\preceq \mathbf{n}\ \text{ and } \mathbf{m}\neq \mathbf{n},
\end{align}
where $\mathbf{n}=(n_1,\cdots,n_N)$.
We define the minimal resonant indices by
\begin{align}\label{eq:defRmin}
\mathbf{R}_{\mathrm{min}}:=\{\mathbf{m} \in \mathbf{R}\ |\ \not\exists \mathbf{n}\in \mathbf{R}\ \mathrm{s.t.}\ |\mathbf{n}|\prec |\mathbf{m}|\}.
\end{align}
We also consider $\mathbf{NR}_1$  formed by the nonresonant indices not larger than   resonant indices:
\begin{align}\label{eq:defnr1}
\mathbf{NR}_1:=\{\mathbf{m}\in \mathbf{NR}\ |\ \forall \mathbf{n}\in \mathbf{R}_{\mathrm{min}},\ |\mathbf{n}| \not \prec |\mathbf{m}|\}.
\end{align}
%
\begin{lemma}\label{lem:RminNR1isfinite}
	Both $\mathbf{R}_{\mathrm{min}}$ and $\mathbf{NR}_1$ are finite sets.
\end{lemma}
For the proof see Appendix \ref{append:1}.

We constructively define functions $\{G_\mathbf{m}\}_{\mathbf{m} \in \mathbf{R}_{\mathrm{min} }}\subset \Sigma^\infty$ which will be important in our analysis.

For $\mathbf{m}\in \mathbf{NR}_1$, we inductively define $\widetilde{\phi}_{\mathbf{m}}(0)$ and $g_{\mathbf{m}}(0)$ by
\begin{align}\label{eq:indefphigroot}
\widetilde{\phi}_{\mathbf{e}_j}(0):=\phi_j,\ g_{\mathbf{e}_j}(0)=0,\ j=1,\cdots,N,
\end{align}
and, for $\mathbf{m}\in \mathbf{NR}_1\setminus \mathbf{NR}_0$,  by
\begin{align}
\widetilde{\phi}_{\mathbf{m}}(0)&:=-(H-\mathbf{m} \cdot \boldsymbol{\omega})^{-1} g_{\mathbf{m}}(0),\label{eq:indefphi}\\
g_{\mathbf{m}}(0)&:=\sum_{m=1}^\infty \frac{1}{m!}g^{(m)}(0)\sum_{(\mathbf{m}_1,\cdots,\mathbf{m}_{2m+1})\in A(m,\mathbf{m})}\widetilde\phi_{\mathbf{m}_1}(0)\cdots \widetilde{\phi}_{\mathbf{m}_{2m+1}}(0),\label{eq:indefg}
\end{align}
where
\begin{align}\label{eq:defAmm}
A(m,\mathbf{m}):=\left\{ \{\mathbf{m}_j\}_{j=1}^{2m+1}\in (\mathbf{NR}_1)^{2m+1}\ |\ \sum_{j=0}^m \mathbf{m}_{2j+1}-\sum_{j=1}^m \mathbf{m}_{2j}= \mathbf{m},\ \sum_{j=0}^{2m+1} |\mathbf{m}_{j}|=|\mathbf{m}|\right\} 
\end{align}

\begin{remark}
	For each $m\geq 1$ and $\mathbf{m} \in \mathbf{NR}_1$, $A(m,\mathbf{m})$ is a finite set.
	Furthermore, for sufficiently large $m$, we have $A(m,\mathbf{m})=\emptyset$.
	Thus, even though we are expressing $g_\mathbf{m}(0)$ in \eqref{eq:indefg} by a series, the sum is finite.
\end{remark}

For $\mathbf{m}\in \mathbf{R}_{\mathrm{min}}$, we define $G_{\mathbf{m}}$ by
\begin{align}\label{eq:defG}
G_{\mathbf{m}}:=\sum_{m=1}^\infty \frac{1}{m!}g^{(m)}(0)\sum_{(\mathbf{m}_1,\cdots,\mathbf{m}_{2m+1})\in A(m,\mathbf{m})}\widetilde\phi_{\mathbf{m}_1}(0)\cdots \widetilde{\phi}_{\mathbf{m}_{2m+1}}(0).
\end{align}
\begin{remark}
	$g_{\mathbf{m}}(0)$ and $G_\mathbf{m}$ are defined similarly. We are using a different notation to emphasize that $g_{\mathbf{m}}(0)$ has
	$\mathbf{m}\in \mathbf{NR}_1$, while $G_{\mathbf{m}}$ has  $\mathbf{m}\in \mathbf{R}_{\mathrm{min}} $.
\end{remark}

The following is the nonlinear Fermi Golden Rule (FGR) assumption.
\begin{assumption}\label{ass:FGR}
	For all $\mathbf{m} \in \mathbf{R}_{\mathrm{min}}$, we assume
	\begin{align}\label{eq:FGR}
	\int_{|\zeta|^2=\mathbf{m}\cdot \boldsymbol{\omega}}|  \widehat{G}_{\mathbf{m}}(\zeta)|^2\,dS \neq 0,
	\end{align}
	where $\widehat{G}_{\mathbf{m}}$ is the distorted Fourier transform associated to $H$.
\end{assumption}

\begin{remark}  In the case $N=2$ and $\omega_1+2(\omega_2-\omega_1)>0$, we have $G_{\mathbf{m}}=g'(0)\phi_1\phi_2^2$, which corresponds to the condition in Tsai and Yau \cite{TY1}, based on the explicit formulas in Buslaev and Perelman  \cite{BP2} and Soffer and Weinstein  \cite{SW3}.
These works are related to  Sigal  \cite{Sigal93CMP}.
 More general situations are considered in
   \cite{CM15APDE}, where however
 the
 $G_{\mathbf{m}}$ are obtained  after a certain number of coordinate changes, so that  the relation of the  $G_{\mathbf{m}}$  and the  $\phi_j$'s is not discussed in   \cite{CM15APDE} and is not easy to track.
\end{remark}

For a generic nonlinear function $g$ the condition \eqref{eq:FGR} is a consequence of the following simpler one, which is similar to (11.6) in Sigal \cite{Sigal93CMP},
\begin{align}\label{eq:FGRsimplified}
	\int_{|\zeta|^2=\mathbf{m}\cdot \boldsymbol{\omega}}| \widehat{\phi ^ \mathbf{m}} (\zeta)|^2\,dS \neq 0 \text{ for all $\mathbf{m} \in \mathbf{R}_{\mathrm{min}}$ }
	\end{align}
where $ \phi ^ \mathbf{m}:= \prod _{j=1,..., N}\phi _j ^{m_j}$.  Both conditions \eqref{eq:FGR} and, even more so, \eqref{eq:FGRsimplified}
are simpler than the analogous conditions in Cuccagna and Maeda \cite{CM15APDE}.

 We have the following.
\begin{proposition} \label{lem:generic g}
Let $\displaystyle L=\sup \{  \frac{\| \mathbf{m} \| -1}{2} : \mathbf{m} \in \mathbf{R}_{\mathrm{min}} \} $ and suppose that the operator $H$
satisfies condition \eqref{eq:FGRsimplified}.  Then there exists an open dense
subset $\Omega $ of $\R ^{L-1}$ s.t.\ if $(g' (0),...., g ^{(L)} (0))\in \Omega $ such that Assumption \ref{ass:FGR} is true   for
\eqref{nls}.
\end{proposition}
\proof See Sect. \ref{append:1}. \qed

For $\mathbf{z}=(z_1,\cdots,z_N)\in \C^N$, $\mathbf{m}=(m_1,\cdots,m_N)\in \Z^N$, we define
\begin{align}\label{eq:zkakko}
\mathbf{z}^\mathbf{m}&:=z_1^{(m_1)}\cdots z_N^{(m_N)} \in \C,\text{ where } z^{(m)}:=\begin{cases} z^m & m\geq 0\\ \bar z^{-m} & m<0,\end{cases}\quad \text{ and }\\
\label{eq:defzabs}
|\mathbf{z}|^k&:=(|z_1|^k,\cdots,|z_N|^k)\in \R^N,\ \|\mathbf{z}\|:=\sum |\mathbf{z}|= \sum_{j=1}^N|z_j|\in\R.
\end{align}

We will use the following notation for a ball in a Banach space $B$:
\begin{align}\label{eq:def:ball}
\mathcal{B}_B(u,r):=\{v\in B\ |\ \|v-u\|_B<r\}.
\end{align}

The ``refined profile" is of the form $\phi(\mathbf{z}) = \mathbf{z} \cdot \boldsymbol{\phi} +o(\|\mathbf{z}\|)$ and is defined by the
following proposition.

\begin{proposition}[Refined Profile]\label{prop:rp}
For any $s \geq 0$, there exist $\delta_s>0$ and $C_s>0$ s.t.\ $\delta_s$ is nonincreasing w.r.t.\ $s\geq 0$ and there exist
\begin{align*}
\{\psi_\mathbf{m}\}_{\mathbf{m}\in \mathbf{NR}_1} &\in C^\infty ( \mathcal{B}_{\R^N}(0,\delta_s^2),(\Sigma^s)^{\sharp \mathbf{NR}_1}), \ \boldsymbol{\varpi}(\cdot) \in C^\infty (\mathcal{B}_{\R^N}(0,\delta_s^2),\R^N) \\
&\text{ and }\mathcal R \in C^\infty(\mathcal{B}_{\C^N}(0,\delta_s),\Sigma^s),
\end{align*}
 s.t.\ $\boldsymbol{\varpi}(0,\cdots,0)=\boldsymbol{\omega}$, $\psi_{\mathbf{m}}(0)=0$ for all $\mathbf{m}\in \mathbf{NR}_1$ and
\begin{align}\label{est:R}
\|\mathcal R(\mathbf{z})\|_{\Sigma^s}\leq C_s \|\mathbf{z} \|^2\sum_{\mathbf{m}\in \mathbf{R}_{\min}} |\mathbf{z}^{\mathbf{m}}|,
\end{align}
where $B_X(a,r):=\{u\in X\ |\ \|u-a\|_X<r\}$, and  if we set
\begin{align}\label{eq:Phianz}
\phi(\mathbf{z}):=\mathbf{z}\cdot\boldsymbol{\phi} + \sum_{\mathbf{m}\in \mathbf{NR}_1}\mathbf{z}^{\mathbf{m}}\psi_{\mathbf{m}}(|\mathbf{z}|^2)\text{ and }z_j(t)=e^{-\im \varpi_j(|\mathbf{z}|^2) t}z_j,
\end{align}
then, setting $\mathbf{z}(t)=(z_1(t),\cdots,z_n(t))$, the function $u(t):=\phi\(\mathbf{z}(t)\)$ satisfies
\begin{align}\label{eq:nlsforce}
\im \partial_t u = H u + g(|u|^2)u -\sum_{\mathbf{m}\in \mathbf{R}_{\mathrm{min}}}\mathbf{z}(t)^{\mathbf{m}} G_{\mathbf{m}} - \mathcal R(\mathbf{z}(t)),
\end{align}
where $\{G_\mathbf{m}\}_{\mathbf{R}_{\min}} \subset\(\Sigma^\infty\)^{\sharp \mathbf{R}_{\min}}$ is given in \eqref{eq:defG}.
Finally, writing $\psi_\mathbf{m}=\psi_\mathbf{m}^{(s)}$, $\boldsymbol{\varpi}=\boldsymbol{\varpi}^{(s)}$ and $\mathcal R=\mathcal R^{(s)}$, for $s_1<s_2$ we have $\psi_\mathbf{m}^{(s_1)}(|\cdot|^2)=\psi_\mathbf{m}^{(s_2)}(|\cdot|^2)$, $\boldsymbol{\varpi}^{(s_1)}(|\cdot|^2)=\boldsymbol{\varpi}^{(s_2)}(|\cdot|^2)$ and $\mathcal R^{(s_1)}=\mathcal R^{(s_2)}$ in $ \mathcal{B}_{\R^N}(0,\delta_{s_2})$.
\end{proposition}
\proof See Sect. \ref{sec:ref_profile}. \qed

The refined profile  $\phi(\mathbf{z})$  contains as a special case the small standing waves bifurcating from the eigenvalues, when they are simple.

\begin{corollary}\label{cor:smallbddst}
Let $s>0$ and $j\in \{1,\cdots,N\}$.
Then, for $z\in \mathcal B_{\C}(0,\delta_s)$, $\phi\(z(t)\mathbf{e}_j\)$ solves \eqref{nls} for $z(t)=e^{-\im \varpi_j(|z\mathbf{e}_j|^2)t}z$.

\end{corollary}

\begin{proof}
Since $(z\mathbf{e}_j)^{\mathbf{m}}=0$ for $\mathbf{m}\in \mathbf{R}_{\mathrm{min}}$, we see that from \eqref{est:R} and \eqref{eq:nlsforce} the remainder terms $\sum_{\mathbf{m}\in \mathbf{R}_{\mathrm{min}}}\mathbf{z}(t)^{\mathbf{m}} G_{\mathbf{m}} + \mathcal R(\mathbf{z}(t))$ are $0$ in \eqref{eq:nlsforce}.
Therefore, we have the conclusion.
\end{proof}

\begin{remark}
	If the eigenvalues of $H$ are not simple the above   does not hold anymore in general. See Gustafson-Phan \cite{GP11SIMA}.
\end{remark}

We call solitons, or standing waves, the functions
\begin{align}\label{eq:defbddst}
\phi_j(z):=\phi(z \mathbf{e}_j).
\end{align}

%
%
%

The main result, which have first proved in \cite{CM15APDE}   is the following.
\begin{theorem}\label{thm:main}
Under the Assumptions \ref{ass:nonres}, \ref{ass:linearInd} and \ref{ass:FGR}, there exist $\delta_0>0$ and $C>0$ s.t.\ for all $u_0\in H^1$ with $\|u_0\|_{H^1}< \delta_0$, there exists $j\in \{1,\cdots,N\}$, $z\in C^1(\R,\C)$, $\eta_+\in H^1$ and $\rho_+\geq 0$ s.t.
\begin{align*}
\lim_{t\to \infty} \|u(t)- \phi_j(z(t)) - e^{\im t \Delta }\eta_+ \|_{H^1}=0,
\end{align*}
and
\begin{align*}
\lim_{t\to \infty}|z(t)|=\rho_+,\ C^{-1}\|u_0\|_{H^1}^2\leq \rho_+^2 + \|\eta_+\|_{H^1}^2 \leq  C\|u_0\|_{H^1}^2.
\end{align*}
\end{theorem}

%


The organization of the paper is the following.
In the rest of this section, we  outline  the proof of the main theorem (Theorem \ref{thm:main}).
In Section \ref{sec:dar}, we introduce the modulation and Darboux coordinate and compute the Taylor expansion of the energy.
In section \ref{sec:prmain} we prove the main theorem (Theorem \ref{thm:main}).
In section \ref{sec:ref_profile} we prove Proposition \ref{prop:rp}.
In section \ref{sec:Dar1}, we state an abstract Darboux theorem with error estimate and apply it to prove Proposition \ref{prop:Darcor}.
In the appendix of this paper, we prove Lemma \ref{lem:RminNR1isfinite}.

We now  outline   the proof of Theorem \ref{thm:main}.
First of all,  the fact that NLS \eqref{nls} is   Hamilton  is crucial.
Indeed, when we consider the symplectic form
\begin{align}\label{def:ssymp}
\Omega_0(\cdot,\cdot):=\<\im \cdot,\cdot\>,  \ \<u,v\>:=\mathrm{Re}  (u,\overline{v}) \text{ where }  (u, {v}):= \int_{\R^3}u(x) {v(x)}\,dx,
\end{align}
and the energy (Hamiltonian) by
\begin{align}\label{def:energ}
E(u)=\frac{1}{2}\<Hu,u\> + \frac{1}{2}\int_{\R^3}G(|u(x)|^2)\,dx,
\end{align}
where $G(s):=\int_0^s g(s)\,ds$, we can rewrite NLS \eqref{nls} as
\begin{align*}
\partial_t u = X_E^{(0)}(u).
\end{align*}
Here, for $F\in C^1(H^1,\R)$,  $X_F^{(0)}$ is the Hamilton vector field of $F$ associated to the symplectic form $\Omega_0$ defined, for $DF$ is the Fr\'echet derivative of $F$, by
\begin{align*}
\Omega_0(X_F^{(0)},\cdot) = DF.
\end{align*}
Next, as usual for the study of stability of solitons, we give a modulation coordinates in $H^1$ in the neighborhood of $0$.
In this paper, we use
\begin{align}\label{eq:modnew}
(\mathbf{z},\eta)\mapsto u=\phi(\mathbf{z})+\eta,
\end{align}
while in \cite{CM15APDE}
we were using
\begin{align}\label{eq:modold}
(\mathbf{z},\eta)\mapsto u=\sum_{j=1,...,N} \phi_j(z_j)+R(\mathbf{z})\eta,
\end{align}
for specific near identity operator $R(\mathbf{z})$  which was first introduced in \cite{GNT}.
Here, in both \eqref{eq:modnew} and \eqref{eq:modold}, $\eta$ is taken from the continuous component of $H$.
That is, $P_c\eta =\eta$,  where
\begin{align}\label{eq:contproj}
P_c u:= u-\sum_{j=1,...,N}\(\<u,\phi_j\>\phi_j+\<u,\im \phi_j\> \im\phi_j\).
\end{align}
The difference between the two coordinates \eqref{eq:modnew} and \eqref{eq:modold} is that in \eqref{eq:modnew} we are using the refined profile which takes into account the nonlinear interactions within the  discrete modes. While the discrete part in
\eqref{eq:modnew} is more complicate than in \eqref{eq:modold}, to prove Theorem \ref{thm:main} for $N>1$ we do not need the $R(\mathbf{z})$ in front of $\eta$.


Unfortunately, even though $\Omega_0$ is a deceptively simple symplectic form, in the coordinates \eqref{eq:modnew} it is complicated (it is very complicated also using coordinates \eqref{eq:modold}).
We thus  introduce a new symplectic form
\begin{align}\label{eq:omega1}
\Omega_1(\cdot,\cdot):=\Omega _0(D_{\mathbf{z}}\phi (\mathbf{z}) D\mathbf{z}\ \cdot ,D_{\mathbf{z}}\phi (\mathbf{z}) D\mathbf{z}\ \cdot ) + \Omega _0(D\eta\ \cdot,D\eta \cdot) ,
\end{align}
which is equal to $\Omega_0$ at $u=0$.
Here, $D_{\mathbf{z}}$ is the Fr\'echet derivative w.r.t.\ the $\mathbf{z}$ variable.

By Darboux theorem  there exists near 0  an almost  identity coordinate change   $\varphi$ such that $\Omega_1=\varphi^* \Omega_0$.
In Sect. \ref{sec:Dar1}  we give a  rather simple proof of  the  type of Darboux theorem needed, viewing it in an   abstract framework
simplifying the analogous part of \cite{CM15APDE}.

For $K=\varphi^* E$, the system becomes
\begin{align*}
\im \partial_t\mathbf{z} =(1+\mathcal{O}(\|\mathbf{z}\|^2))\nabla_\mathbf{z}K,\quad \im \partial_t \eta =\nabla_\eta K,
\end{align*}
where $\nabla_{\mathbf{z}}$ and $\nabla_\eta$ are the gradient corresponding to the Fr\'echet derivative w.r.t.\ $\mathbf{z}$ and $\eta$.
In the new coordinates, the energy $K$  expands
\begin{align*}
K = E(\phi(\mathbf{z}))+E(\eta)+\<\widetilde{\mathcal{R}}(\mathbf{z}),\eta\>+\mathrm{error}.
\end{align*}
When using  the coordinate system \eqref{eq:modold}, in order to estimate the solutions it is necessary
like in     \cite{CM15APDE}  to make further normal forms changes of variables. But using coordinates \eqref{eq:modnew}
we are ready for the estimates and there is no need of normal forms. First of all,  we have $\widetilde{\mathcal{R}}(\mathbf{z}) = \sum_{\mathbf{m}\in \mathbf{R}_{\mathrm{min}}}\mathbf{z}^{\mathbf{m}}G_{\mathbf{m}}+\mathrm{error}$, see  the First Cancelation Lemma, Lemma \ref{lem:cancel}.
This implies that
\begin{align}\label{eq:etaintro}
\im \partial_t\eta =H\eta +P_c g(|\eta|^2)\eta +\sum_{\mathbf{m}\in \mathbf{R}_{\mathrm{min}}}\<\mathbf{z}^{\mathbf{m}}G_{\mathbf{m}},\eta\> +\mathrm{error}.
\end{align}
Thus, by the endpoint Strichartz estimate, to show that $\eta$ scatters it suffices to show $\mathbf{z}^\mathbf{m} \in L^2(\R ) $     for $\mathbf{m}\in \mathbf{R}_{\mathrm{min}}$. To check this point, we consider
\begin{align*}
\frac{d}{dt}E(\phi(\mathbf{z}))=\sum_{\mathbf{m}\in \mathbf{R}_{\mathrm{min}}}\left\{E(\phi(\mathbf{z})),\<\mathbf{z}^{\mathbf{m}}G_{\mathbf{m}},\eta\>\right\}+\mathrm{error},
\end{align*}
where $\{\cdot,\cdot\}$ is the Poisson bracket associated to $\Omega_1$.  We obtain
\begin{align} \left\{E(\phi(\mathbf{z})),\<\mathbf{z}^{\mathbf{m}}G_{\mathbf{m}},\eta\>\right\} = (\boldsymbol{\omega}\cdot \mathbf{m})\<\im \mathbf{z}^{\mathbf{m}}G_{\mathbf{m}},\eta\>+\mathrm{error}, \label{miracle}
\end{align}
where, see below \eqref{eq:FGR:poisson1} and as a consequence of the Second Cancelation Lemma, Lemma \ref{lem:cancel2},
\begin{align*}| \mathrm{error} |\lesssim |\mathbf{z}|   \sum_{\mathbf{m}\in\mathbf{R}_{\mathrm{min}}}  |\mathbf{z}^\mathbf{m} |  \text{ for all $|\mathbf{z}|\le 1$}.
\end{align*}
Notice that   $z_1 ^\ell$ does not satisfy this inequality no matter how large we take $\ell \in \mathbb{N}$, so the error term in
\eqref{miracle} is not just small, but has a specific combinatorial structure.  In \cite{CM15APDE}, to get the structure \eqref{eq:etaintro} and
to bound $\mathbf{z}$, a painstaking normal forms argument was required, but here these fact come for free.

From this point on, the proof ends in a standard way.  Since    $\eta\sim -\mathbf{z}^\mathbf{m}(H-\omega\cdot \mathbf{m}-\im 0)^{-1}G_\mathbf{m}$,
 where the latter  is the solution of \eqref{eq:etaintro} without the nonlinear term and "error", we have, omitting errors
\begin{align*}
\frac{d}{dt}E(\phi(\mathbf{z}))=\sum_{\mathbf{m}\in \mathbf{R}_{\mathrm{min}}} (\boldsymbol{\omega}\cdot \mathbf{m})|\mathbf{z}^\mathbf{m}|^2\<\im G_{\mathbf{m}},(H-\omega\cdot \mathbf{m}-\im 0)^{-1}G_\mathbf{m}\>.
\end{align*}
Since $\<\im G_{\mathbf{m}},(H-\omega\cdot \mathbf{m}-\im 0)^{-1}G_\mathbf{m}\>$ equals  \eqref{eq:FGR} in Assumption \ref{ass:FGR} which we have assumed positive, this above idealized identity yields
\begin{align*}
 E(\phi(\mathbf{z}(t)))+ \sum_{\mathbf{m}\in \mathbf{R}_{\mathrm{min}}} \|  \mathbf{z}^\mathbf{m} \| ^2 _{L^2(0,t)}\le  E(\phi(\mathbf{z}(0))).
\end{align*}
Using this, we can close estimates.

We conclude with a few comments on   refined profiles,  which  play a central role in our proof. One of the distinctive features
of our system is the existence or non existence of small quasi--periodic solutions which are not periodic. Sigal \cite{Sigal93CMP}
stated their absence, and this follows from \cite{CM15APDE} and our analysis here. The $\mathbf{z}^{\mathbf{m}}G_{\mathbf{m}}$ terms
in $\widetilde{\mathcal{R}}(\mathbf{z})$  are resonant, cannot be eliminated from the equation exactly if \eqref{eq:FGR} holds and are an obstruction to the existence of quasi--periodic solutions. On the other hand, there are no resonant terms in the  discrete NLS with $N=2$, where  quasi-periodic solutions are proved to exist in    Maeda \cite{Maeda17SIMA}. Furthermore,  in Maeda \cite{Maeda17SIMA}  an equivalence  is observed  between being able to see  quasi--periodic solutions, absence of resonant terms in the equations and, finally,  existence of coordinate systems where the mixed term $\langle\widetilde{\mathcal{R}}(\mathbf{z}),\eta\rangle$, that is nonlinear degree 1 in $\eta$,   is absent from
the energy.
Our main insight here is that,  since there are no small  quasi--periodic solutions, we might try to replace them with a     surrogate  (refined profiles),  in the  expectation  of    an equivalence, analogous to that considered in  Maeda \cite{Maeda17SIMA}, between
this surrogate   and optimal coordinate systems.
This works and,
while in  \cite{CM15APDE} we searched directly, and with great effort, for the coordinates, here
we find, with a relatively elementary method, the refined profiles. Starting from  the  refined profiles  we define
a natural coordinate system. It turns out that these coordinates are optimal, as is seen in elementary fashion noticing
that the fact that the refined profiles are approximate solutions  of   \eqref{nls}, specifically they solve   \eqref{eq:nlsforce},   provides us the two Cancelation Lemmas, which in turn guarantee that our coordinates are optimal.
We end remarking that  refinements of the ansatz were already   in the great series by Merle and Raphael \cite{ MR03GAFA,MR04,MR05AM,MR4}, which has   inspired   our   notion of refined profile.


\section{Darboux coordinate and Energy expansion}\label{sec:dar}

We start from constructing the modulation coordinate.
First, we have the following.
\begin{lemma}\label{lem:lincor}
For any $s\in \R$ there exist  $\delta_s>0$ and $\mathbf z\in C^\infty (   \mathcal{B}_{  \Sigma^{-s}}  (0, \delta_s)  , \C^N)$ s.t.
\begin{align*}
u-\phi(\mathbf{z}(u)) \in P_c\Sigma^{-s},
\end{align*}
where $P_c$ is given by \eqref{eq:contproj}.
\end{lemma}

\begin{proof}
	This is an immediate consequence of the implicit function theorem. We consider
	\begin{align*} F_j (\mathbf{z},u)=
	\< \phi(\mathbf{z} ) -u, \phi _j\> + \im \< u-\phi(\mathbf{z} ) , \im \phi _j\>  \text{ for }j=1,...,N.
	\end{align*}
	We have $F:=( F_1,...,F_N)\in C^\infty ( \Sigma^{-s} \times \mathcal{B} _{\C ^N}(0,\delta _0)  , \C ^{N})$  for $\delta _0>0$ given in Proposition \ref{prop:rp}.  Obviously $\left . F \right | _{(z,u)=(0,0)} =0$ and from
	$ \psi_{\mathbf{m}}(0)=0$ for all $ \mathbf{m}\in \mathbf {NR}_1$ , it follows $\left . D_\mathbf{z} F \right | _{(\mathbf{z},u)=(0,0)} = \mathrm{Id}_{\C ^N}$, where $D_\mathbf{z}F$ is the Fr\'echet derivative w.r.t.\ the $\mathbf{z}$ variable. By implicit function theorem
	we obtain the desired $\mathbf z\in C^\infty (   \mathcal{B}_{  \Sigma^{-s}}  (0, \delta_s)  , \C^N)$ for some $\delta_s>0$.
\end{proof}

By Lemma \ref{lem:lincor}, we have our first (modulation) coordinate.
\begin{proposition}\label{prop:lincor}
For any $s\in \R$ there exist  $\delta_s>0$ s.t.\ the map
\begin{align}\label{eq:modcor}
\mathcal{B}_{\C^N}(0,\delta_s)\times \mathcal{B}_{P_c X^{-s}}(0,\delta_s) \ni (\mathbf{z},\eta)\mapsto \phi(\mathbf{z})+\eta \in X^{-s}, \ X^s=\Sigma^s\text{ or }H^s,
\end{align}
is a $C^\infty$ local diffeomorphism.
Moreover, we have
\begin{align*}
\|u\|_{X^s}\sim_s \|\mathbf{z}\| + \|\eta\|_{X^s}.
\end{align*}
\end{proposition}

\begin{proof}
	It is an direct consequence of Lemma \ref{lem:lincor}.
\end{proof}

For Banach spaces $X,Y$, we set $\mathcal L(X,Y)$ to be the Banach space of all bounded linear operators from $X$ to $Y$.
Moreover, we set $\mathcal L(X):=\mathcal L(X,X)$.

For $F\in C^1(\mathcal{B}_{H^1}(0,\delta),\R)$, we write
\begin{align*}
F(\mathbf{z},\eta):=F(\phi(\mathbf{z})+\eta).
\end{align*}
We define $D_\eta F(\mathbf{z},\eta) \in C(\mathcal{B}_{H^1}(0,\delta),\mathcal L(P_c H^1,\R))$ and $\nabla_\eta F(\mathbf{z},\eta)\in C(\mathcal{B}_{H^1}(0,\delta),P_c H^{-1})$ by
\begin{align*}
\forall Y\in P_c H^1,\ D_\eta F(\mathbf{z},\eta)Y=\<\nabla_\eta F(\mathbf{z},\eta),Y\>:=\left.\frac{d}{d \epsilon}\right|_{\epsilon=0}F(\mathbf{z},\eta+\epsilon v).
\end{align*}
Here, for Banach spaces $A,B$, $\mathcal L(A,B)$ is the Banach space of all bounded operators from $A$ to $B$.
Similarly, we define $\nabla_{\mathbf{z}} F(u)=\nabla_{\mathbf{z}} F(\mathbf{z},\eta) \in C(\mathcal{B}_{H^1}(0,\delta),\C^N)$ by
\begin{align*}
\forall \mathbf w\in \C^N,\ \<\nabla_{\mathbf{z}} F(\mathbf{z},\eta), \mathbf w\>_{\C^N} :=D_\mathbf{z}F(\mathbf{z},\eta)\mathbf{w}=\left.\frac{d}{d \epsilon}\right|_{\epsilon=0}F(\mathbf{z}+\epsilon \mathbf w, \eta),
\end{align*}
where $\<\mathbf{w}_1,\mathbf{w}_2\>_{\C^N}=\Re \sum_{j=1}^Nw_{1j}\overline{w_{2j}}$ for $\mathbf{w}_k=(w_{k1},\cdots,w_{kN})$.

Using the above notations, for $u\in B_{H^1}(0,\delta)$ and $Y\in H^1$, we have
\begin{align}\label{eq:chain}
DF(\mathbf{z},\eta)Y =\<\nabla_{\mathbf{z}}F(\mathbf{z},\eta), D\mathbf{z}Y\>_{\C^N}+D_\eta F(\mathbf{z},\eta) D\eta Y,
\end{align}
where $D\mathbf{z}$ and $D\eta$ are Fr\'echet derivatives of functions $\mathbf{z}(u)$, $\eta(u):=u-\phi(\mathbf{z}(u))$.

Notice that, since the Fr\'echet derivative of the identity map $u\mapsto u$ is an identity, we have
\begin{align}\label{eq:id}
\mathrm{Id}_{X^s}=Du=D_{\mathbf{z}}\phi (\mathbf{z}) D\mathbf{z} +D\eta.
\end{align}

\begin{remark}
	Even though $\eta =P_c \eta$,
	$D\eta$ is not $P_c$ except at $u=0$.
\end{remark}

By \eqref{eq:id}, we have
\begin{align*}
\Omega_0=\Omega _0(D_{\mathbf{z}}\phi (\mathbf{z}) D\mathbf{z} ,D_{\mathbf{z}}\phi (\mathbf{z}) D\mathbf{z} ) + \Omega _0(D\eta ,D\eta )   + \Omega_0 (D_{\mathbf{z}}\phi (\mathbf{z}) D\mathbf{z} ,D\eta ) + \Omega _0 (D\eta , D_{\mathbf{z}}\phi (\mathbf{z}) D\mathbf{z} )  .
\end{align*}
Therefore, removing the cross terms (the latter two terms), we have the symplectic form $\Omega_1$ given in \eqref{eq:omega1}.
Given $F \in C^1(\mathcal{B}_{H^1}(0,\delta),\R)$, the Hamilton vector field $X_F^{(1)}$ associated to the symplectic form $\Omega_1$ is defined by $\Omega_1(X_F^{(1)},\cdot)=DF$.
Thus, by \eqref{eq:chain}, we have
\begin{align}\label{eq:sym11}
\<\im D_{\mathbf{z}}\phi (\mathbf{z}) D\mathbf{z} X_F^{(1)},D_{\mathbf{z}}\phi (\mathbf{z}) D\mathbf{z} Y\>+\<\im D\eta X_F^{(1)},D\eta Y\>=\<\nabla_{\mathbf{z}}F, D\mathbf{z}Y\>_{\C^N}+D_\eta F D\eta Y.
\end{align}
In particular, we have
\begin{align}\label{Hamv1eta}
\im D\eta X_F^{(1)}=\nabla_\eta F.
\end{align}
We turn to $\mathbf{z}$. Setting $\psi(\mathbf{z}):=\sum_{\mathbf{m} \in \mathbf{NR}_1}\mathbf{z}^\mathbf{m}\psi_\mathbf{m}(|\mathbf{z}|^2)$, we have $\phi(\mathbf{z})=\mathbf{z}\cdot \boldsymbol{\phi}+\psi(\mathbf{z})$ with $\|\psi(\mathbf{z})\|_{\Sigma^s}\lesssim_s \|\mathbf{z}\|^3$.
Then, since $\nabla_{\mathbf{z}}\phi(\mathbf{z})\mathbf{w}=\mathbf{w}\cdot\boldsymbol{\phi}+\mathcal{O}_{\mathcal{L}(\C^N,\Sigma^s)}(\|\mathbf{z}\|^2)\mathbf{w}$, $\<\im \mathbf{w}_1\cdot\boldsymbol{\phi}, \mathbf{w}_2\cdot\boldsymbol{\phi}\>=\<\im \mathbf{w}_1,\mathbf{w}_2\>_{\C^N}$ and $\mathcal L(\C^N\times \C^N,\R)\simeq \mathcal L(\C^N)$, we see there exists $\widetilde{A} \in C^\infty(\mathcal{B}_{\C^N}(0,\delta_0),\mathcal L(\C^N))$ s.t.
\begin{align*}
\<\im D_{\mathbf{z}}\phi (\mathbf{z}) D\mathbf{z} X_F^{(1)},D_{\mathbf{z}}\phi (\mathbf{z}) D\mathbf{z} Y\>=\< \im\( 1+ \widetilde{A}(\mathbf{z}) \)D\mathbf{z} X_F^{(1)},D\mathbf{z} Y\>_{\C^N},
\end{align*}
with $\|\widetilde{A}(\mathbf{z})\|_{\mathcal L(\C^N)}\lesssim \|\mathbf{z}\|^2$.
Thus, setting $A \in C^\infty(\mathcal{B}_{\C^N}(0,\delta_0),\mathcal L(\C^N))$ by $1+A(\mathbf{z})=(1+\widetilde{A}(\mathbf{z}))^{-1}$, we have $\|A(\mathbf{z})\|_{\mathcal{L}(\C^N)}\lesssim \|\mathbf{z}\|^2$ and
\begin{align}\label{eq:odez}
\im D\mathbf{z}X_F^{(1)}=(1+A(\mathbf{z}))\nabla_{\mathbf{z}}F.
\end{align}

The following proposition allows us to move to the ``diagonalized" symplectic form $\Omega_1$.

\begin{proposition}\label{prop:Darcor}
For any  $s>0$  there exists $\delta_s>0$   and  $\varphi\in C^\infty(\mathcal{B}_{\Sigma^{-s}}(0,\delta_s),\Sigma^{-s})$   satisfying
\begin{align}\label{eq:Darcor1}
\|\varphi(u)-u\|_{\Sigma^s} \le C_s \|\mathbf{z}(u)\|^2 \|\eta(u)\|_{\Sigma^{-s}}
\end{align}
which is  a  local diffeomorphism  and such that
\begin{align*}
\varphi^* \Omega _0=\Omega_1.
\end{align*}
\end{proposition}

We give the proof of Proposition \ref{prop:Darcor} in section \ref{sec:Dar1}.
It will be a direct consequence of an abstract Darboux theorem with error estimate (Proposition \ref{prop:darboux}).

We study the dynamics of $u^*=\varphi^{-1}(u)$, where $u$ is the solution of NLS \eqref{nls} with $\|u(0)\|_{H^1}\ll1$, which reduces to the study of the dynamics of $\mathbf{z}(u^*)$ and $\eta(u^*)$.
Since $u(t)$ is the integral curve of the Hamilton vector field $X_E^{(0)}$, $u^*(t)$ is the integral curve of the Hamilton vector field $X_K^{(1)}$, where $K:=\varphi^*E=E(\varphi(\cdot))$.
By \eqref{Hamv1eta}, \eqref{eq:odez}, we have
\begin{align}\label{eq:etaz}
\im \partial_t \eta =\nabla_\eta K(\mathbf{z},\eta),\quad  \im \partial_t \mathbf{z} =(1+A(\mathbf{z}))\nabla_{\mathbf{z}} K(\mathbf{z},\eta).
\end{align}
To compute the r.h.s.\ of \eqref{eq:etaz}, we expand $K$.
Before going into the expansion, we prepare a notation  to denote some reminder terms.
\begin{definition}\label{def:err_1}
	Let $F \in C^1(\mathcal{B}_{H^1}(0,\delta) ,\R)$ for some $\delta>0$.
	We write $F=\mathcal R_1$ if, for $s\geq 0$, there exists $\delta_s>0$ s.t.\ for $\|u\|_{H^1}<\delta_s$  we have
	\begin{align}
	\| \nabla_\eta F(u)\|_{\Sigma^s} + \| \nabla_{\mathbf{z}} F (u) \|&\lesssim_s  \|u\|_{H^1}^2\(\|\eta\|_{\Sigma^{-s}}+\sum_{\mathbf{m}\in \mathbf{R}_{\mathrm{min}}}|\mathbf{z}^{\mathbf{m}}|\).\label{eq:def:err_1}
	\end{align}
	In our notation, if $F=\mathcal R_1$ and $G=\mathcal R_1$, we will have $F+G=\mathcal R_1$.
	So, an equation like $F+\mathcal R_1 =\mathcal R_1$ will not mean $F=0$ but only $F=\mathcal R_1$.
	This rule will also be applied to $\mathcal{R}_2$ below.
\end{definition}

By Taylor expanding $F(s,t)= K(s\mathbf{z},t\eta ) $, we have
\begin{align} &
K(\mathbf{z},\eta )=K(0,\eta)+K(\mathbf{z},0)+  \int _0^1\partial_s\partial_t  K(s\mathbf{z},0)\,ds+\int_0^1\int_0^1 (1-t) \partial_s\partial_t^2 K(s\mathbf{z},t\eta)\,dtds    .\label{eq:exp_K1}
\end{align}
Since $\varphi (\eta )=\eta$ by \eqref{eq:Darcor1}, we have $   K(0,\eta )= E(\eta)$. Similarly, since $\varphi(\phi(\mathbf{z}))=\phi(\mathbf{z})$, we have $K(\mathbf{z},0)= E(\phi(\mathbf{z}))$.
The third term of the r.h.s.\ of \eqref{eq:exp_K1} is
\begin{align*}
\int _0^1\partial_s\partial_t  K(s\mathbf{z},0)\,ds=\partial_tK(\mathbf{z},0)=\<\nabla_\eta K(\mathbf{z},0),\eta\>,
\end{align*}
because $D_\eta K(0,0)=0$.
The following lemma is the crux   of this paper.
\begin{lemma}[First Cancellation Lemma]\label{lem:cancel}
We have, near the origin,
\begin{align}\label{eq:cancel1}
\nabla_\eta K \(\mathbf{z},0\)=P_c D\varphi(\phi(\mathbf{z}))^* \(\sum_{\mathbf{m}\in \mathbf{R}_{\mathrm{min}}}\mathbf{z} ^{\mathbf{m}} G_{\mathbf{m}}+\mathcal R(\mathbf{z} )\).
\end{align}
\end{lemma}

\begin{proof}
We fix arbitrary $\mathbf{z}_0=(z_{01},\cdots,z_{0N}) \in \mathcal{B} _{\C^N} (0, \delta _0 )$ with $\delta_0$ sufficiently small.
It is enough to prove \eqref{eq:cancel1} with $\mathbf{z}=\mathbf{z}_0$.
We  set $\mathbf{z}_0(t)=(z_{01}(t),\cdots,z_{0N}(t))\in \C^N$ with$$z_{0j}(t)=e^{-\varpi_j(|\mathbf{z}_0|^2)t}z_{0j},$$
 where $\varpi_j$ is also given in Proposition \ref{prop:rp}.
Consider the non-autonomous Hamiltonian
\begin{align*}
E_{\mathbf{z}_0}(u,t):=E(u)-\sum_{\mathbf{m}\in \mathbf{R}_{\mathrm{min}}}\<\mathbf{z}_0(t) ^{\mathbf{m}} G_{\mathbf{m}},u\>-\<\mathcal R(\mathbf{z}_0(t)),u\>.
\end{align*}
Then,  the Hamilton vector field $X_{E_{\mathbf{z}_0}}^{(0)}(u,t)$ of  $E_{\mathbf{z}_0}(u,t)$ associated with the symplectic form $\Omega_0$ is
\begin{align*}
\im X_{E_{\mathbf{z}_0}}^{(0)}(u,t) =     H u + g(|u|^2)u -\sum_{\mathbf{m}\in \mathbf{R}_{\mathrm{min}}}\mathbf{z}_0(t) ^{\mathbf{m}} G_{\mathbf{m}} -\mathcal R(\mathbf{z}_0(t)) \ .
\end{align*}
Thus,  by Proposition  \ref{prop:rp},   $\phi(\mathbf{z}_0(t))$ is   the  integral curve of this flow with initial value  $ \phi(\mathbf{z}_0)$. 

\noindent Consider now  the pullback of $E_{\mathbf{z}_0}(u,t)$ by  the $\varphi$ of Proposition \ref{prop:Darcor}.
By Taylor expansion we get
\begin{align*}
&\varphi^* E_{\mathbf{z}_0}( u,t) = K(u) -\<\sum_{\mathbf{m}\in \mathbf{R}_{\mathrm{min}}}\mathbf{z}_0(t) ^{\mathbf{m}} G_{\mathbf{m}}+\mathcal R(\mathbf{z}_0(t)),\varphi(u)\> = \\&
K(u)-\<\sum_{\mathbf{m}\in \mathbf{R}_{\mathrm{min}}}\mathbf{z}_0(t) ^{\mathbf{m}} G_{\mathbf{m}}+\mathcal R(\mathbf{z}_0(t)),\phi(\mathbf{z})+D\varphi(\phi(\mathbf{z}))\eta+\int_0^1(1-s) D^2\varphi(\phi(\mathbf{z}+s\eta))(\eta,\eta)\>.
\end{align*}
Differentiating in $\eta $ at $\eta =0$, yields
\begin{align*}
&\left . \nabla_\eta  \( \varphi^* E_{\mathbf{z}_0}(  t) \) \right | _{\eta =0}= \left . \nabla_\eta K\right | _{\eta =0}  - P_c (D\varphi(\phi(\mathbf{z}))) ^*\(\sum_{\mathbf{m}\in \mathbf{R}_{\mathrm{min}}}\mathbf{z}_0(t) ^{\mathbf{m}} G_{\mathbf{m}}+\mathcal R(\mathbf{z}_0(t))\) .
\end{align*}
Because of \eqref{eq:Darcor1}, we know that  $\varphi^{-1} (\phi(\mathbf{z}))=\phi(\mathbf{z})$  for all $\mathbf{z}$.  Then, $\phi(\mathbf{z}_0)$  is an integral trajectory
also for $\varphi^* E_{\mathbf{z}_0}( u,t)$.  But since, in $(\mathbf{z},\eta )$, integral trajectories satisfy $\im \dot \eta =  \nabla_\eta \( \varphi^* E_{\mathbf{z}_0}(  t) \)$, form $\eta \equiv 0$ and thus from $\dot \eta \equiv 0$,  it follows
that $\left . \nabla_\eta \( \varphi^* E_{\mathbf{z}_0}(  t) \)\right | _{\eta =0} \equiv 0$.  So, for $t=0$, we obtain \eqref{eq:cancel1}.
\end{proof}

By Proposition \ref{prop:Darcor}, Definition \ref{def:err_1} and \eqref{eq:cancel1}, we have
\begin{align}\label{eq:eta1st}
\<\nabla_\eta K(\mathbf{z},0),\eta\>=\sum_{\mathbf{m}\in \mathbf{R}_{\mathrm{min}}}\<\mathbf{z} ^{\mathbf{m}} G_{\mathbf{m}},\eta\>+\mathcal{R}_1.
\end{align}
We next study the last term in r.h.s.\ of \eqref{eq:exp_K1}.
By direct computation, for the linear part of the energy  we have
\begin{align*}
&\partial_s\partial_t^2\<H\varphi(s\phi(\mathbf{z})+t\eta), \varphi(s\phi(\mathbf{z})+t\eta)\>=
 4\<HD^2\varphi(s\phi(\mathbf{z})+t\eta)(\phi,\eta),D\varphi(s\phi(\mathbf{z})+t\eta)\eta\>
 \\&\quad
 +2\<HD^2\varphi(s\phi(\mathbf{z})+t\eta)(\eta,\eta),D\varphi(s\phi(\mathbf{z})+t\eta)\phi\>
 +2\<HD^3\varphi(s\phi(\mathbf{z})+t\eta)(\phi,\eta,\eta),\varphi(s\phi(\mathbf{z})+\eta)\>.
\end{align*}
Thus,
\begin{align}\label{eq:enegykin2}
\frac{1}{2}\int_0^1\int_0^1 (1-t)  \partial_s\partial_t^2\<H\varphi(s\phi(\mathbf{z})+t\eta), \varphi(s\phi(\mathbf{z})+t\eta)\>\,dtds=\mathcal{R}_1.
\end{align}

For the nonlinear part of the energy, we have
\begin{align}
&\partial_s\partial_t^2\int_{\R^3}G(| u_{t,s} |^2)\,dx=
4\<2g'' u_{t,s}\(\mathrm{Re}\(\overline{u_{t,s}}\ \widetilde\eta\)\)^2+2g' \widetilde\eta \mathrm{Re}\(\overline{u_{t,s}}\ \widetilde\eta\)+g'u_{t,s}|\widetilde\eta|^2,\widetilde\phi\>\nonumber
\\&
\quad +2\<2g'u_{t,s}\mathrm{Re}\(u_{t,s}\overline{D^2\varphi(\eta,\eta)}\)+g D^2\varphi(\eta,\eta),\widetilde\phi\>
+4\<2g'u_{t,s}\mathrm{Re}\(\overline{u_{t,s}}\ \widetilde\eta\)+g \widetilde\eta,D^2\varphi (\phi(\mathbf{z}),\eta)\>\nonumber\\&
\quad +2\<g u_{t,s},D^3\varphi(\phi(\mathbf{z}),\eta,\eta)\>.\label{eq:expandR2}
\end{align}
where $u_{t,s}:=\varphi(s\phi(\mathbf{z})+t\eta)$, $\widetilde \eta=D\varphi(s\phi(\mathbf{z})+t\eta)\eta$, $\widetilde\phi=D\varphi(s\phi(\mathbf{z}+t\eta))\phi(\mathbf{z})$, $g^{(k)}=g^{(k)}(|u_{t,s}|^2)$ and $D^{k+1}\varphi=D^{k+1}\varphi(s\phi(\mathbf{z})+t\eta)$  for $k=0,1,2$.

To handle these terms, we introduce another notation of error terms.

\begin{definition}\label{def:err_2}
	Let $\delta>0$ and $F\in C^3(\mathcal{B}_{H^1}(0,\delta),\R)$.
	We write $F=\mathcal R_2$ if $F$  is a linear combination of  functions of the form
	\begin{align*}
	\int_0^1\int_0^1(1-t)\<f(u_{t,s}), \mathfrak{f}(u_{t,s})(\phi,\eta,\eta) \>\,dtds,
	\end{align*}
	where $f(u)(x)=\tilde f(\Re u(x),\Im u)$ with $\tilde f\in C^\infty(\R^2,\C)$  and where either one or the other of the following
two conditions are satisfied:
	\begin{itemize}
		\item[(I)] $|\tilde f(s_1,s_2)|\lesssim |s|\<s\>^2$, $|\partial_{s_j}\tilde f(s_1,s_2)|\lesssim \<s\>^2$ ($j=1,2$), $|\partial_{s_j}\partial_{s_k}\tilde f(s_1,s_2)|\lesssim \<s\> $ ($j,k=1,2$) and  $\mathfrak{f}(u)(\phi,\eta,\eta):=\(D\varphi(u)\phi\)\(D\varphi(u)\eta\)^2$;
		\item[(II)] $|\tilde f(s_1,s_2)|\lesssim |s|^2\<s\>^2$, $|\partial_{s_j}\tilde f(s_1,s_2)|\lesssim |s|\<s\>^2$ ($j=1,2$), $|\partial_{s_j}\partial_{s_k}\tilde f(s_1,s_2)|\lesssim \<s\>^2$ ($j,k=1,2$) and  $\mathfrak{f}(u)(\phi,\eta,\eta):=\(D\varphi(u)\phi\) D^2\varphi(u)(\eta,\eta)$ or $D\varphi(u)\eta D^2\varphi(u)(\phi,\eta)$ or $D^3\varphi(u)(\phi,\eta,\eta)$.
	\end{itemize}
Here, $s=(s_1,s_2)$ and $|s|=(s_1^2+s_2^2)^{1/2}$, $\<s\>=(1+s_1^2+s_2^2)^{1/2}$.
%
%
%
%
%
%
\end{definition}
Thus, we have
\begin{align}\label{eq:energynp}
\frac{1}{2}\int_0^1\int_0^1(1-t)\partial_s\partial_t^2\int_{\R^3}G(|\varphi(s\phi(\mathbf{z})+\eta)|^2)\,dx=\mathcal R_2.
\end{align}
We record that under the assumption $\|u\|_{H^1}\lesssim 1$, we have
\begin{align}\label{eq:est:DzR2}
\|\nabla_{\mathbf{z}}\mathcal{R}_2\|\lesssim \|u\|_{H^1}\|\eta\|_{L^6}^2.
\end{align}

Summarizing, \eqref{eq:exp_K1}, \eqref{eq:eta1st}, \eqref{eq:enegykin2} and \eqref{eq:energynp} we have
\begin{align}\label{eq:exp_K}
K(u)=E(\phi(\mathbf{z}))+E(\eta) + \sum_{\mathbf{m}\in \mathbf{R}_{\mathrm{min}}}\<\mathbf{z} ^{\mathbf{m}} G_{\mathbf{m}},\eta\>+\mathcal R_1+\mathcal R_2.
\end{align}

We can study the structure of $E(\phi(\mathbf{z}))$ by an argument similar to the proof of Lemma \ref{lem:cancel}.

\begin{lemma}[Second Cancellation Lemma]\label{lem:cancel2}
	We have
	\begin{align}
	(1+A(\mathbf{z}))\nabla_{\mathbf{z}}E(\phi(\mathbf{z}))=\Lambda(|\mathbf{z}|^2)\mathbf{z}+B(\mathbf{z}),\label{eq:cancel21}
	\end{align}
	where, $\Lambda(|\mathbf{z}|^2)\mathbf{w}:=(\varpi_1(|\mathbf{z}|^2)w_1,\cdots,\varpi_N(|\mathbf{z}|^2)w_N)$ and $\|B(\mathbf{z})\|\lesssim \sum_{\mathbf{m}\in\mathbf{R}_{\mathrm{min}}}|\mathbf{z}^\mathbf{m}|$.
\end{lemma}

\begin{proof}
	Fix $\mathbf{z}_0\in \mathcal{B}_{\C^N}(0,\delta_0)$ and consider  $\mathbf{z}_0(t)$  and $E_{\mathbf{z}_0}(u,t)$
as in the proof of Lemma \ref{lem:cancel}. Then $(\mathbf{z}_0(t),0)$  is an integral curve of $\varphi ^*E_{\mathbf{z}_0}(u,t)$  and
	  for $t=0$ we have
	\begin{align*}
	\Lambda(|\mathbf{z}_0 |^2)\mathbf{z}_0 &=(1+A(\mathbf{z}_0 ))\left.\nabla_{\mathbf{z}}\right|_{\mathbf{z}=\mathbf{z}_0 ,\eta=0,t=0 } \( \varphi ^*E_{\mathbf{z}_0}(u,t) \)\\&=
	(1+A(\mathbf{z}_0))\left. \nabla_\mathbf{z}\right|_{\mathbf{z}=\mathbf{z}_0}\(E(\phi(\mathbf{z}))-
	\<\sum_{\mathbf{m}\in\mathbf{R}_{\mathrm{min}}}\mathbf{z}_0^{\mathbf{m}}G_\mathbf{m}+\mathcal{R}(\mathbf{z}_0),\phi(\mathbf{z})\>
	\).
	\end{align*}
This yields the equality \eqref{eq:cancel21}  at $\mathbf{z}=\mathbf{z}_0$ with the desired bound on the remainder term, thanks to
\begin{align*} &
	\left \|\left.\nabla_\mathbf{z}\right|_{\mathbf{z}=\mathbf{z}_0}\<\sum_{\mathbf{m}\in\mathbf{R}_{\mathrm{min}}}
\mathbf{z}_0^{\mathbf{m}}G_\mathbf{m}+\mathcal{R}(\mathbf{z}_0),\phi(\mathbf{z})\> \right \|   = \left \|\<\sum_{\mathbf{m}\in\mathbf{R}_{\mathrm{min}}}
\mathbf{z}_0^{\mathbf{m}}G_\mathbf{m}+\mathcal{R}(\mathbf{z}_0),\left.\nabla_\mathbf{z}\right|_{\mathbf{z}=\mathbf{z}_0}\phi(\mathbf{z})\> \right \|
\\& \le  \left \| \sum_{\mathbf{m}\in\mathbf{R}_{\mathrm{min}}}
\mathbf{z}_0^{\mathbf{m}}G_\mathbf{m}+\mathcal{R}(\mathbf{z}_0)\right  \| _{L^2(\R ^3) }   \| \left.\nabla_\mathbf{z}\right|_{\mathbf{z}=\mathbf{z}_0}\phi(\mathbf{z})  \| _{L^2 (\R ^3)}
\lesssim \sum_{\mathbf{m}\in\mathbf{R}_{\mathrm{min}}}|\mathbf{z}_0^\mathbf{m}|.
	\end{align*}
\end{proof}

\section{Proof of the main theorem}\label{sec:prmain}

Given an interval $I\subseteq \R$ we set
\begin{align*}
\mathrm{Stz}^j(I):=L^\infty_t  H^j (I)   \cap L^2_t  W^{j,6}(I),\quad \mathrm{Stz}^{*j}(I):=L^1_t  H^j(I )  +  L^2_t W^{j,6/5}(I ),\ j=0,1,
\end{align*}
where $H^0=L^2$ and $W^{0,p}=L^p$.
We will be using the Strichartz inequality, see  \cite{Y1}:
\begin{align*}
\|e^{-\im t H}P_c v\|_{\mathrm{Stz}^j}\lesssim \|v\|_{H^j},\ \|\int_0^t e^{-\im(t-s)H}f(s)\,ds\|_{\mathrm{Stz}^j}\lesssim \|f\|_{\mathrm{Stz}^{*j}}, \ j=0,1 .
\end{align*}
We now consider the Hamiltonian system   in the $(\mathbf{z},\eta )$ with Hamiltonian $K$ and symplectic form $\Omega _1$. Then we have the following.

 \begin{theorem}[Main Estimates]\label{thm:mainbounds}
There exist $\delta _0>0$ and $C_0>0$ s.t.\ if the constant  $\|u_0\|_{H^1}< \delta_0$      for $I= [0,\infty )$ and $C=C_0$    we have:
\begin{align}
   \|  \eta \| _{\mathrm{Stz}^1(I)} +\sum_{\mathbf{m}\in \mathbf{R}_{\mathrm{min}} }\|  \mathbf{z}^{\mathbf{m}} \| _{L^2_t(I)} &\le
  C   \|u_0\|_{H^1},
  \label{Strichartzradiation}\\
\| \mathbf{z}  \|
  _{W ^{1,\infty} _t  (I )} &\le
  C   \|u_0\|_{H^1}.
   \label{L^inftydiscrete}
\end{align}
Furthermore,  there exists $\rho  _+\in [0,\infty )^N$ s.t.\   there exist  a   $j_0$  with $\rho_{+j}=0$ for $j\neq j_0$,
and   there exists $\eta _+\in H^1$
with $\|  \eta _+\| _{H^1}\le C    \epsilon $  for $C=C_0$, such that
\begin{equation}\label{eq:small en31}
\begin{aligned}&     \lim_{t\to +\infty}\| \eta (t )-
e^{\im t\Delta }\eta  _+     \|_{H^1 }=0  \quad  , \quad
  \lim_{t\to +\infty} |z_j(t)|  =\rho_{+j}  .
\end{aligned}
\end{equation}

\end{theorem}

Note that from the energy and mass conservation, Definitions \ref{def:err_1} and \ref{def:err_2}, \eqref{eq:est:DzR2}, \eqref{eq:exp_K} and Lemma \ref{lem:cancel2} and  we have the apriori bound
\begin{align*}
\|\mathbf{z}\|_{W^{1,\infty}_t(\R)}+\|\eta\|_{L^\infty_t H^1(\R)}\lesssim \|u_0\|_{H^1}.
\end{align*}

The proof that Theorem \ref{thm:mainbounds}  implies Theorem \ref{thm:main} is like in \cite{CM15APDE}.
Furthermore, by completely routine arguments  discussed in \cite{CM15APDE},
\eqref{Strichartzradiation} for  $I= [0,\infty )$ is a consequence of the following Proposition.

\begin{proposition}\label{prop:mainbounds} There exists  a  constant $c_0>0$ s.t.\
for any  $C_0>c_0$ there is a value    $\delta _0= \delta _0(C_0)   $ s.t.\  if    \eqref{Strichartzradiation}
holds  for $I=[0,T]$ for some $T>0$, for $C=C_0$  and for $u_0\in B_{H^1}(0,\delta_0)$,
then in fact for $I=[0,T]$  the inequalities  \eqref{Strichartzradiation} holds  for   $C=C_0/2$.
\end{proposition}

The rest of this section is devoted to the proof of Proposition \ref{prop:mainbounds}.
In the following, we always assume \eqref{Strichartzradiation} holds for $C=C_0$ and the integration w.r.t.\ $t$ is always be over $I$.

We fist estimate the contribution of $\mathcal{R}_j$, $j=1,2$.

%
%
%
%
%

\begin{lemma}\label{lem:stzR12}
	Under the assumption of Proposition \ref{prop:mainbounds},   there is a constant $ C(C_0)$ such that  \begin{align*}
	\|\nabla_\eta \mathcal{R}_j \|_{\mathrm{Stz}^{*1}}\le  C(C_0) \|u_0\|_{H^1}^3,\quad j=1,2.
	\end{align*}
\end{lemma}

\begin{proof}
	For $\mathcal{R}_1$, we have
	\begin{align*}
	\| \nabla_\eta\mathcal{R}_1 \|_{\mathrm{Stz}^{*1}}\leq \| \nabla_\eta\mathcal{R}_1 \|_{L^2_t \Sigma^1(I)} \lesssim \|u_0\|_{H^1}^2 \(\|\eta\|_{\mathrm{Stz}^1(I)}+\sum_{\mathbf{m}\in \mathbf{R}_{\min}}\|\mathbf{z}^{\mathbf{m}}\|_{L^2_t(I)}\)\lesssim C_0\|u_0\|_{H^1}^3.
	\end{align*}
	We next estimate type (I) of $\mathcal{R}_2$.
	Ignoring the integral w.r.t.\ $t$ and $s$ and the complex conjugate, which are irrelevant in the estimate, we have
	\begin{align}\label{eq:DR21}
	D_\eta \mathcal{R}_2w = \< f'(u)w,\widetilde\phi \widetilde{\eta}^2\>+\<f(u),D^2\varphi(u)(\phi,w)\widetilde{\eta}^2+2\widetilde\phi \widetilde{\eta}D^2\varphi(u)(\eta,w)+2\widetilde{\phi}\widetilde{\eta}w\>.
	\end{align}
	where $f'(u)w=\partial_R f(u)\Re w+\partial_I f(u)\Im w$ and $\widetilde{\phi}$, $\widetilde{\eta}$ are defined in \eqref{eq:expandR2}.
	The contribution of the first term in the r.h.s.\ of $\eqref{eq:DR21}$ can be estimated as
	\begin{align}\label{eq:DR211}
	\|f'(u)\widetilde{\phi}\widetilde{\eta}^2\|_{L^2_tL^{6/5}}\lesssim \|\mathbf{z}\|_{L^\infty_t}\|\eta\|_{L^\infty_t L^6}\|\eta\|_{L^2_t L^6} \lesssim C_0^3 \|u_0\|_{H^1}^3,
	\end{align}
	where we have used $\|\<u\>\|_{L^\infty+L^6}\lesssim 1$ and the Sobolev embedding $H^1\hookrightarrow L^6$.
	Furthermore, 
	\begin{align}\label{eq:DR212}
	\|\nabla_x\(f'(u)\widetilde{\phi}\widetilde{\eta}^2\)\|_{L^2_tL^{6/5}}\lesssim \|f''(u)\nabla_x u \widetilde{\phi}\widetilde{\eta}^2\|_{L^2_tL^{6/5}}+\|f'(u)\nabla_x\widetilde{\phi}\widetilde{\eta}^2\|_{L^2_tL^{6/5}}+\|f'(u)\widetilde{\phi}\widetilde{\eta}\nabla_x \widetilde\eta\|_{L^2_tL^{6/5}},
	\end{align}
	and, using Sobolev's embedding $W^{1,6}\hookrightarrow L^\infty $,
	\begin{align*}
	\|f''(u)\nabla_x u \widetilde{\phi}\widetilde{\eta}^2\|_{L^2_tL^{6/5}}\lesssim \|\<u\>\widetilde{\phi}\|_{L^\infty_t L^6}\|\nabla_x u\|_{L^\infty_t L^2}\|\eta\|_{L^2_tL^\infty} \|\eta\|_{L^\infty_t L^6}\lesssim C_0^3\|u_0\|_{H^1}^3.
	\end{align*}
	Similar estimates hold for the other two terms in \eqref{eq:DR212}.
	
	Turning to the contribution of the second term in \eqref{eq:DR21}, we have
	\begin{align*}
	\sup_{\|w\|_{(W^{1,6/5})^*}\leq 1} |\<f(u),D^2\varphi(u)(\phi,w)\widetilde{\eta}^2\>|&\lesssim \|f(u){\eta}^2\|_{\Sigma^{-1}} \|\phi\|_{\Sigma^{-1}}\|w\|_{\Sigma^{-1}}\lesssim \|\mathbf{z}\| \|f(u){\eta}^2\|_{L^{6/5}} \\&
	\lesssim \|\mathbf{z}\| \|u\|_{L^2\cap L^6}\|\<u\>\|_{L^6+L^\infty}^2\|\eta\|_{L^6}^2.
	\end{align*}
	where we have used $(W^{1,6/5})^*\hookrightarrow \Sigma^{-1}$ and $L^{6/5}\hookrightarrow \Sigma^{-1}$ which hold by duality.
	Thus, we have the estimate $\lesssim C_0\|u_0\|_{H^1}^3$ for this term too.
	The third term in \eqref{eq:DR21} can be estimated just as the second term and the fourth term can be estimated just as the first term.
	
	The estimates of the type (II) terms in $\mathcal R_2$ is similar, easier and is omitted.
\end{proof}

From
\begin{align*}
\im \partial_t\eta = \nabla_\eta K(u)=P_c\( H \eta +g(|\eta|^2)\eta + \sum_{\mathbf{m}\in \mathbf{R}_{\mathrm{min}}}\mathbf{z} ^{\mathbf{m}} G_{\mathbf{m}} +\nabla_\eta \mathcal R_1 + \nabla_\eta \mathcal R_2\),
\end{align*}
by Lemma \ref{lem:stzR12}  we obtain
\begin{align}\label{eq:esteta}
\|\eta\|_{\mathrm{Stz}^1}\lesssim \|u_0\|_{H^1} + C(C_0)\|u_0\|_{H^1}^3 + \sum_{\mathbf{m}\in \mathbf{R}_{\mathrm{min}}} \|\mathbf{z} ^{\mathbf{m}}\|_{L^2_t}.
\end{align}
We need   bounds on $\mathbf{z} $.
  We set $Z:= \sum_{\mathbf{m}\in \mathbf{R}_{\min}} \mathbf{z}^{\mathbf{m}}R_+(\mathbf{m}\cdot \boldsymbol{\omega})P_c G_{\mathbf{m}}$ and $\xi :=\eta+Z$, where $R_+(\lambda):=(H-\lambda-\im 0)^{-1}$.
Then, 
\begin{align*}
\im \partial_t \xi=P_c\(H\xi + g(|\eta|^2)\eta + \nabla_\eta \mathcal R_1 + \nabla_\eta \mathcal R_2+\mathcal{R}_3\),
\end{align*}
where $\mathcal R_3:=\im \partial_t Z - HZ + \sum_{\mathbf{m}\in \mathbf{R}_{\mathrm{min}}}\mathbf{z} ^{\mathbf{m}} P_cG_{\mathbf{m}}$, which satisfies
\begin{align*}
\mathcal R_3=\sum_{\mathbf{m}\in \mathbf{R}_{\min}}  a_{\mathbf{m}} R_+(\mathbf{m}\cdot \boldsymbol{\omega})P_c G_{\mathbf{m}},\ \text{where }   a_{\mathbf{m}}:=\im \partial_t (\mathbf{z}^\mathbf{m})-(\mathbf{m}\cdot\boldsymbol{\omega}) \mathbf{z}^\mathbf{m}.
\end{align*}

\begin{lemma}\label{lem:esta}
	Under the assumption of Proposition \ref{prop:mainbounds},   there is a constant $ C(C_0)$ such that
	\begin{align}
	\|a_{\mathbf{m}}\|_{L^2_t(I)} \le C(C_0)\|u_0\|_{H^1}^3.\label{eq:aest11}
	\end{align}
\end{lemma}

\begin{proof}
	We have $a_\mathbf{m}=\sum_{j=1}^N a_{\mathbf{m},j}$ with
\begin{equation}\label{eq:aest12}
 a_{\mathbf{m},j}= \left\{\begin{array}{l}
 m_j(\im \partial_t z_j-\omega_j z_j)\frac{\mathbf{z}^\mathbf{m}}{z_j} \text{  if $m_j>0$}
\\[1ex]
m_j\overline{(\im \partial_t z_j-\omega_j z_j)}\frac{\mathbf{z}^\mathbf{m}}{\overline{z_j}} \text{  if $m_j<0$.}
\end{array}
\right.
\end{equation}
	By \eqref{eq:odez}, \eqref{eq:exp_K} and Lemma \ref{lem:cancel2}, we have
	\begin{align}
	&\im \partial_t z_j -\omega_j z_j
	=\(\im \partial_t \mathbf{z} - \Lambda(0)\mathbf{z}\)\cdot \mathbf{e}_j \label{eq:aest1}\\
	&=\(\(\Lambda(|\mathbf{z}|^2)-\Lambda(0)\)\mathbf{z}+B(\mathbf{z}) +(1+A(\mathbf{z}))\nabla_{\mathbf{z}}\(\sum_{\mathbf{R}_{\mathrm{min}}}\<\mathbf{z}^\mathbf{m}G_{\mathbf{m}},\eta\>+\mathcal{R}_1+\mathcal{R}_2\)\)\cdot \mathbf{e}_j.\nonumber
	\end{align}
	We estimate  each $ a_{\mathbf{m},j}$ by distinguishing the contribution coming from the terms in the last line in \eqref{eq:aest1}.

\noindent Using $\(\Lambda(|\mathbf{z}|^2)-\Lambda(0)\)\mathbf{z} \cdot \mathbf{e}_j=(\varpi_j(|\mathbf{z}|^2)-\omega_j)z_j$,
	for the first term   we have
	\begin{align*}&
	m _j \|\(\Lambda(|\mathbf{z}|^2)-\Lambda(0)\)\mathbf{z} \cdot \mathbf{e}_j\frac{\mathbf{z}^{\mathbf{m}}}{z_j}\|_{L^2_t} =m_j \| (\varpi_j(|\mathbf{z}|^2)-\omega_j)\mathbf{z}^{\mathbf{m}}\|_{L^2_t}\\&  \lesssim  \|\mathbf{z} \|_{L^\infty _t} ^2  \|\mathbf{z}^{\mathbf{m}}\|_{L^2_t}  \le C_0 ^3\|u_0\|_{H^1}^3.
	\end{align*}
Similarly, by Lemma \ref{lem:cancel2},
\begin{align*}&
	m _j \|B (\mathbf{z}) \cdot \mathbf{e}_j\frac{\mathbf{z}^{\mathbf{m}}}{z_j}\|_{L^2_t } \lesssim   \sum_{\mathbf{n}\in\mathbf{R}_{\mathrm{min}}} m _j \|\mathbf{z}^\mathbf{n}\frac{\mathbf{z}^{\mathbf{m}}}{z_j}\|_{L^2_t }  \le
\sum_{\mathbf{n}\in\mathbf{R}_{\mathrm{min}}} m _j  \|\mathbf{z}^\mathbf{n} \|_{L^2_t }  \|  \frac{\mathbf{z}^{\mathbf{m}}}{z_j}\|_{L^\infty_t } \lesssim C_0 ^3\|u_0\|_{H^1}^3,
	\end{align*}
from the fact that   $\mathbf{m}\in \mathbf{R}_\mathrm{min}$implies $\|\mathbf{m}\|\geq 3$.

\noindent For    $ \mathbf{n}\in \mathbf{R}_{\mathrm{min}} $,  we have
\begin{align*}&
	m _j \|(1+A(\mathbf{z}))\nabla_{\mathbf{z}}\<\mathbf{z}^\mathbf{n}G_{\mathbf{m}},\eta\>\cdot \mathbf{e}_j\frac{\mathbf{z}^{\mathbf{m}}}{z_j}\|_{L^2_t } \lesssim     m _j  \| \eta \| _{L^2_tL^6_x}
\| \nabla_{\mathbf{z}} \mathbf{z}^\mathbf{n}  \|_{L^\infty _t }    \|  \frac{\mathbf{z}^{\mathbf{m}}}{z_j}\|_{L^\infty_t } \lesssim C_0 ^5\|u_0\|_{H^1}^5.
	\end{align*}
Similar estimates using \eqref{eq:def:err_1} and  \eqref{eq:est:DzR2} can be obtained for the terms with $\mathcal{R}_1$ and $ \mathcal{R}_2$.
\end{proof}

When we seek for the nonlinear effect of the  radiation $\eta$ on the $\mathbf{z}$,     we think of  $Z$ as the main term and  of $\xi$ as a   remainder term.
We first estimate $\xi$.

\begin{lemma}\label{lem:estxiL2}
	Under the assumption of Proposition \ref{prop:mainbounds},   there is a constant $ C(C_0)$ such that
	\begin{align*}
	\|\xi\|_{L^2_t\Sigma^{0-}} \lesssim \|u_0\|_{H^1}+C(C_0)\|u_0\|_{H^1}^3.
	\end{align*}
	
\end{lemma}

Here, the the key difference from \eqref{eq:esteta} is that the last summation in the r.h.s. of \eqref{eq:esteta} has been eliminated.
This because the formula $\xi=-Z +\eta $  is a normal form expansion designed exactly to eliminate that summation from the equation of $\xi$.

\begin{proof}

Since $\xi=\eta+Z$, we have
\begin{align*}
\|\xi\|_{L^2_t\Sigma^{0-}}\lesssim & \|e^{-\im t H}\eta(0)\|_{\mathrm{Stz}^0}+\|e^{-\im t H}Z(0)\|_{L^2_t\Sigma^{0-}}+\|g(|\eta|^2)\eta\|_{\mathrm{Stz}^{*0}} + \|\nabla_\eta \mathcal{R}_1\|_{\mathrm{Stz}^{*0}}\\&+\|\nabla_\eta\mathcal{R}_2\|_{\mathrm{Stz}^{*0}}+\|\int_0^t e^{-\im (t-s)H}P_c \mathcal{R}_3\,ds\|_{L^2_t\Sigma^{0-}}.
\end{align*}
Using the estimate $\|e^{-\im t H}R_+(\mathbf{m}\cdot \boldsymbol{\omega})P_cf\|_{\Sigma^{0-}}\lesssim\<t\>^{-3/2}\|f\|_{\Sigma^0}$ for $\mathbf{m}\in \mathbf{R}_{\mathrm{min}}$, we have
\begin{align*}
\|e^{-\im t H}Z(0)\|_{L^2_t\Sigma^{0-}}&\lesssim \sum_{\mathbf{m}\in \mathbf{R}_{\mathrm{min}} } \|\mathbf{z}(0)\|^{\|\mathbf{m}\|}\lesssim \|u_0\|_{H^1}^3 \ \text{and}\\
\|\int_0^t e^{-\im (t-s)H}P_c \mathcal{R}_3\,ds\|_{L^2_t\Sigma^{0-}}&\lesssim \|a_{\mathbf{m}}\|_{L^2_t}\lesssim C(C_0)\|u_0\|_{H^1}^3
\end{align*}
Therefore, we have the conclusion.
\end{proof}

We recall that for $F,G\in C^1(B_{H^1}(0,\delta),\R)$ we have the    Poisson brackets   given by
\begin{align*}
\{F,G\}:=DF X_G^{(1)}=\Omega_1(X_F^{(1)},X_G^{(1)}).
\end{align*}
Obviously $\{F,G\}=-\{G,F\}$.
The relevance here is that, if $u(t)$ is an integral curve of the Hamilton vector field $X_G^{(1)}$,  then $\frac{d}{dt}F(u(t))=\{F,G\}$.
Therefore
\begin{align}
  \frac{d}{dt}E(\phi(\mathbf{z}))=\{ E(\phi(\mathbf{z})), K(\mathbf{z},\eta)\}
   =\left\{ E(\phi(\mathbf{z})), \sum_{\mathbf{m}\in \mathbf{R}_{\mathrm{min}}}\< \mathbf{z}^{\mathbf{m}}G_\mathbf{m},\eta\> +\mathcal R_1 +\mathcal R_2\right\} ,\label{eq:FGR:poisson1}
\end{align}
where we used that $\{ E(\phi(\mathbf{z})), E(\phi(\mathbf{z}))\}=\{E(\phi(\mathbf{z})), E(\eta) \}=0$ because Poisson brackets are anti-symmetric and the symplectic form is diagonal w.r.t.\ $\mathbf{z}$ and $\eta$.
For the main   Poisson bracket in the r.h.s.  \eqref{eq:FGR:poisson1} we claim
\begin{align}\nonumber
  \sum_{\mathbf{m}\in \mathbf{R}_{\mathrm{min}}} \left\{E(\phi(\mathbf{z})), \<\mathbf{z}^{\mathbf{m}}G_\mathbf{m}, \eta \> \right\} &= -\sum_{\mathbf{m}\in \mathbf{R}_{\mathrm{min}}} \<\nabla_{\mathbf{z}}\<\mathbf{z}^{\mathbf{m}}G_\mathbf{m}, \eta \>,D\mathbf{z}X_{E(\phi(\mathbf{z}))}^{(1)}\>_{\C^N}\\&
  =\sum_{\mathbf{m}\in \mathbf{R}_{\mathrm{min}}}\< \im (\boldsymbol{\omega}\cdot\mathbf{m}) \mathbf{z}^{\mathbf{m}} G_{\mathbf{m}},\eta \>+\mathcal{R}_4,\label{eq:FGR:Poisson2}
\end{align}
where $\mathcal{R}_4=\sum_{\mathbf{m}\in \mathbf{R}_{\mathrm{min}} }\<\im \widetilde{a_{\mathbf{m}}}G_{\mathbf{m}},\eta\>$ with $\widetilde{a_{\mathbf{m}}}=D_{\mathbf{z}}(\mathbf{z}^{\mathbf{m}})\((\Lambda(|\mathbf{z}|^2)-\Lambda(0))\mathbf{z}+B(\mathbf{z})\)$.
To prove formula  \eqref{eq:FGR:Poisson2}, using  Lemma \ref{lem:cancel2}
we compute
\begin{align*} &
   \left\{E(\phi(\mathbf{z})), \<\mathbf{z}^{\mathbf{m}}G_\mathbf{m}, \eta \> \right\}  =  - \<\nabla_{\mathbf{z}}\<\mathbf{z}^{\mathbf{m}}G_\mathbf{m}, \eta \>,D\mathbf{z}X_{E(\phi(\mathbf{z}))}^{(1)}\>_{\C^N} \\& =\<\nabla_{\mathbf{z}}\<\mathbf{z}^{\mathbf{m}}G_\mathbf{m}, \eta \>, \im (1+A(z) E(\phi(\mathbf{z}))   \>_{\C^N} =  \<\nabla_{\mathbf{z}}\<\mathbf{z}^{\mathbf{m}}G_\mathbf{m}, \eta \>, \im \Lambda(0)\mathbf{z}   \>_{\C^N} \\& +  \<\nabla_{\mathbf{z}}\<\mathbf{z}^{\mathbf{m}}G_\mathbf{m}, \eta \>, \im \(  (\Lambda(|\mathbf{z}|^2)-\Lambda(0) )\mathbf{z} + B(\mathbf{z})\)   \>_{\C^N}.
\end{align*}
 By   elementary computations, we have the following, which completes the proof of  \eqref{eq:FGR:Poisson2}:
\begin{align*}  &
    \<\nabla_{\mathbf{z}}\<\mathbf{z}^{\mathbf{m}}G_\mathbf{m}, \eta \>, \im \Lambda(0)\mathbf{z}   \>_{\C^N}  = 2 ^{-1}\sum _{j=1,...,N}\<  \nabla_{\mathbf{z}} \mathbf{z}^{\mathbf{m}} ( G_\mathbf{m}, \overline{\eta} )  +  \nabla_{\mathbf{z}} \overline{\mathbf{z} ^{\mathbf{m}}}
    ( \overline{G}_\mathbf{m},  \eta  ), \im \omega _j z_j \mathbf{e}_j   \>_{\C^N}\\& = 2 ^{-1}\sum _{j=1,...,N}
    \left [ \partial _{z_j}  \( \mathbf{z}^{\mathbf{m}} ( G_\mathbf{m}, \overline{\eta} )+   \overline{\mathbf{z} ^{\mathbf{m}}}
    ( \overline{G}_\mathbf{m},  \eta  )     \)   \im \omega _j z_j  - \partial _{\overline{z}_j}  \( \mathbf{z}^{\mathbf{m}} ( G_\mathbf{m}, \overline{\eta} )+   \overline{\mathbf{z} ^{\mathbf{m}}}
    ( \overline{G}_\mathbf{m},  \eta  )     \)   \im \omega _j \overline{z}_j         \right ] \\& = 2 ^{-1} \im   (\boldsymbol{\omega}\cdot\mathbf{m}) \mathbf{z}^{\mathbf{m}}( G_\mathbf{m}, \overline{\eta} ) -  2 ^{-1} \im   (\boldsymbol{\omega}\cdot\mathbf{m}) \overline{\mathbf{z} ^{\mathbf{m}}}
    ( \overline{G}_\mathbf{m},  \eta  )  = \< \im (\boldsymbol{\omega}\cdot\mathbf{m}) \mathbf{z}^{\mathbf{m}} G_{\mathbf{m}},\eta \> .
\end{align*}

\noindent Proceeding as  in  Lemma   \ref{lem:esta} we have
\begin{align}\label{est:atildeL2}
\|\widetilde{a_{\mathbf{m}}}\|_{L^2_t(I)}\le C(C_0)\|u_0\|_{H^1}^3.
\end{align}
Entering  the expansion  $\eta=-Z+\xi$, we obtain
\begin{align}\label{eq:FGR:Poisson3}
  \sum_{\mathbf{m}\in \mathbf{R}_{\mathrm{min}}}\< \im (\boldsymbol{\omega}\cdot\mathbf{m}) \mathbf{z}^{\mathbf{m}} G_{\mathbf{m}},\eta \>=
  -\sum_{\mathbf{m}\in \mathbf{R}_{\mathrm{min}}}(\boldsymbol{\omega}\cdot\mathbf{m})|\mathbf{z}|^{2|\mathbf{m}|}\< \im   P_cG_{\mathbf{m}}, R_+(\mathbf{m}\cdot \boldsymbol{\omega}) P_cG_{\mathbf{m}} \>+\mathcal{R}_5+\mathcal{R}_6,
\end{align}
where
\begin{align*}
\mathcal{R}_5=-\sum_{{\mathbf{m},\mathbf{n}\in \mathbf{R}_{\mathrm{min}},\  \mathbf{m}\neq \mathbf{n}}}\< \im (\boldsymbol{\omega}\cdot\mathbf{m}) \mathbf{z}^{\mathbf{m}} G_{\mathbf{m}}, \mathbf{z}^{\mathbf{n}} R_+(\mathbf{m}\cdot \boldsymbol{\omega}) P_cG_{\mathbf{n}} \>,\quad
\mathcal{R}_6=\sum_{\mathbf{m}\in \mathbf{R}_{\mathrm{min}}}\< \im (\boldsymbol{\omega}\cdot\mathbf{m}) \mathbf{z}^{\mathbf{m}} G_{\mathbf{m}},\xi \>.
\end{align*}

\begin{lemma}\label{lem:estR46}
	We have, for a fixed constant $c_0$
	\begin{align} \sum_{j=1,2}\|\<\partial_\mathbf{z}\mathcal{R}_j,D\mathbf{z}X_{E(\phi(\mathbf{z}))}^{(1)}\>\|_{L^1_t(I)}+\sum_{j=4,6}\|\mathcal{R}_j\|_{L^1_t(I)}\le c_0C_0   \|u_0\| _{H^1}^2   +  C(C_0)   \(\|u_0\| _{H^1}^3+      \|u_0\| _{H^1}^4\)   .\label{eq:estR456}
	\end{align}
\end{lemma}
Here the crucial point is that in the quadratic term we have $C_0$ instead of $C_0^2$, while the exact dependence  in $C_0$ of $ C(C_0)$
is immaterial.
\begin{proof}  The main bound is the following, using   Lemma \ref{lem:estxiL2} and the a priori estimate \eqref{Strichartzradiation},
\begin{align*} \|\mathcal{R}_6\|_{L^1_t(I)} &\lesssim   \|\xi\|_{L^2_t\Sigma^{0-}}   \sum_{\mathbf{m}\in \mathbf{R}_{\mathrm{min}} }\|  \mathbf{z}^{\mathbf{m}} \| _{L^2_t(I)} \lesssim      C_0 \(  \|u_0\|_{H^1}+C(C_0)\|u_0\|_{H^1}^3 \)  \|u_0\|_{H^1}.
\end{align*}
Turning to the remainders, for $j=1 $  (resp. $j=2$) the upper bound $\le C(C_0)   \|u_0\| _{H^1}^3 $ follows from \eqref{eq:def:err_1} (resp. \eqref{eq:est:DzR2}) and the a priori estimates \eqref{Strichartzradiation}. The  upper bound $\le C(C_0)   \|u_0\| _{H^1}^4 $ for $j=4 $ follows from \eqref{est:atildeL2}.
\end{proof}

\begin{lemma}\label{lem:estR5}
	We have
	\begin{align}\label{eq:est:R5}
	\left|\int_I \mathcal R_5\,dt\right|\lesssim C_0^2\|u_0\|_{H^1}^4.
	\end{align}
\end{lemma}

\begin{proof}
	Let $ \mathbf{m}\neq \mathbf{n} $.
	By \eqref{eq:etaz}, \eqref{eq:exp_K} and Lemma \ref{lem:cancel2}, we have
	\begin{align*}
	\im\partial_t(\mathbf{z}^\mathbf{m}\mathbf{z}^{\mathbf{-n}})=(\mathbf{m-n})\cdot \boldsymbol{\omega} \mathbf{z}^\mathbf{m}\mathbf{z}^{\mathbf{-n}} + \mathcal R_7,
	\end{align*}
	where
	\begin{align*}
	\mathcal R_7=&(\mathbf{m-n})\cdot (\boldsymbol{\varpi}(|\mathbf{z}|^2)-\boldsymbol{\omega}) \mathbf{z}^\mathbf{m}\mathbf{z}^{\mathbf{-n}}\\&\quad+D_{\mathbf{z}}(\mathbf{z}^\mathbf{m}\mathbf{z}^{\mathbf{-n}})\(B(\mathbf{z})+(1+A(\mathbf{z}))\(\nabla_{\mathbf{z}}\(\sum_{\mathbf{R}_{\mathrm{min}}}\<\mathbf{z}^\mathbf{m}G_{\mathbf{m}},\eta\>+\mathcal{R}_1+\mathcal{R}_2\)\)\).
	\end{align*}
	Then, we have
	\begin{align*}
	\|\mathcal{R}_7\|_{L^2}\lesssim C_0 \|u_0\|_{H^1}^4.
	\end{align*}
	Therefore, since
	\begin{align*}
	 &\< \im (\boldsymbol{\omega}\cdot\mathbf{m}) \mathbf{z}^{\mathbf{m}} G_{\mathbf{m}},  \mathbf{z}^{\mathbf{n}} R_+(\boldsymbol{\omega}\cdot \mathbf{n}) G_{\mathbf{n}} \>
	\\&\quad=-\frac{\boldsymbol{\omega}\cdot\mathbf{m}}{(\mathbf{m-n})\cdot \boldsymbol{\omega}}\(\partial_t \<   \mathbf{z}^{\mathbf{m}} G_{\mathbf{m}},  \mathbf{z}^{\mathbf{n}} R_+(\boldsymbol{\omega}\cdot \mathbf{n}) G_{\mathbf{n}} \> +\<   \mathcal{R}_7 G_{\mathbf{m}},   R_+(\boldsymbol{\omega}\cdot \mathbf{n}) G_{\mathbf{n}} \>\),
	\end{align*}
	integrating the above equation over $I$, we have \eqref{eq:est:R5}.
\end{proof}

From \eqref{eq:FGR:poisson1}, \eqref{eq:FGR:Poisson2}, \eqref{eq:FGR:Poisson3}, Lemmas \ref{lem:estR46} and \ref{lem:estR5}, and
\begin{align*}
\< \im G_{\mathbf{m}},   (H-\boldsymbol{\omega}\cdot \mathbf{m} -\im 0)^{-1} G_{\mathbf{m}} \> =     \frac{1}{16 \pi \sqrt{\boldsymbol{\omega}\cdot \mathbf{m}}} \int_{|\zeta|^2=\boldsymbol{\omega}\cdot \mathbf{m}}|\widehat{G_{\mathbf{m}}}(\zeta)|\,d\zeta\gtrsim 1,
\end{align*}
( for the latter see
$(H-\boldsymbol{\omega}\cdot \mathbf{m} -\im 0)^{-1}
=\mathrm{P.V.}\frac{1}{H-\boldsymbol{\omega}\cdot \mathbf{m}}+\im \pi \delta(H-\boldsymbol{\omega}\cdot \mathbf{m})
$  and formula (2.5) p. 156 \cite{taylor2})       and Assumption \ref{ass:FGR},
we have
\begin{align}\label{eq:estzm}
\sum_{\mathbf{m}\in \mathbf{R}_{\mathrm{min}} }\|\mathbf{z}^\mathbf{m}\|_{L^2}^2 \lesssim C_0\|u_0\|_{H^1}^2 +C(C_0)\|u_0\|_{L^2}^3.
\end{align}
By taking  $\|u_0\|_{H^1}< \delta_0$ with $\delta_0>0$  sufficiently small, the l.h.s. in \eqref{eq:estzm} is
smaller than $c_0^2 C_0\|u_0\|_{H^1}^2$ for a fixed $c_0$.  Adjusting the constant and using  \eqref{eq:esteta} we conclude that
\eqref{Strichartzradiation} with $C=C_0$ implies
\begin{align*}
   \|  \eta \| _{\mathrm{Stz}^1(I)} +\sum_{\mathbf{m}\in \mathbf{R}_{\mathrm{min}} }\|  \mathbf{z}^{\mathbf{m}} \| _{L^2_t(I)} &\le
  c_0  \sqrt{C_0}   \|u_0\|_{H^1}  < \frac{C_0}{2}\|u_0\|_{H^1}
\end{align*}
where $c_0$ is a fixed constant and we are free to choose $C_0>4c_0^2$, so that the last inequality is true. This completes the proof of Proposition \ref{prop:mainbounds}.
\qed

\section{Soliton and refined profile}
\label{sec:ref_profile}

In this section, we prove Proposition \ref{prop:rp}.
We first note that due to our notation \eqref{eq:zkakko}, $\mathbf{z}^{\mathbf{m}_1}\mathbf{z}^{\mathbf{m}_2}$ is not $\mathbf{z}^{\mathbf{m}_1+\mathbf{m}_2}$ in general. In fact, we have the following elementary lemma.

\begin{lemma}\label{lem:zmmult}
	Let $\mathbf{m}_1,\mathbf{m}_2\in \Z^N$ and $\mathbf{z}\in \C^N$.
	Then,
	\begin{align*}
	\mathbf{z}^{\mathbf{m}_1}\mathbf{z}^{\mathbf{m}_2}=|\mathbf{z}|^{|\mathbf{m}_1|+|\mathbf{m}_2|-|\mathbf{m}_1+\mathbf{m}_2|}\mathbf{z}^{\mathbf{m}_1+\mathbf{m}_2}.
	\end{align*}
\end{lemma}

\begin{proof}
	It suffices to consider   $N=1$, where $\mathbf{m}_1,\mathbf{m}_2\in \Z$. If they are both $\ge 0$   or $\le 0$, then $|\mathbf{m}_1|+|\mathbf{m}_2|-|\mathbf{m}_1+\mathbf{m}_2|=0$ and it is immediate from \eqref{eq:zkakko} that
	$\mathbf{z}^{\mathbf{m}_1}\mathbf{z}^{\mathbf{m}_2}= \mathbf{z}^{\mathbf{m}_1+\mathbf{m}_2}$.
	Otherwise, we reduce to $\mathbf{m}_1>0>\mathbf{m}_2$.
	Then $|\mathbf{m}_1|+|\mathbf{m}_2|-|\mathbf{m}_1+\mathbf{m}_2| =    2  |\mathbf{m}_{j_0}| $     with $|\mathbf{m}_{j_0}|= \min _{j}|\mathbf{m}_j|$. If $j_0=2$, we have $\mathbf{z}^{\mathbf{m}_1}\mathbf{z}^{\mathbf{m}_2} = \mathbf{z}^{\mathbf{m}_1} \bar{\mathbf{z}}^{ |\mathbf{m}_2|} = |\mathbf{z}| ^{2|\mathbf{m}_2|} \mathbf{z}^{\mathbf{m}_1+\mathbf{m}_2}$,
	which is the desired formula.
	If $j_0=1$, then $\mathbf{z}^{\mathbf{m}_1}\mathbf{z}^{\mathbf{m}_2} = \mathbf{z}^{\mathbf{m}_1} \bar{\mathbf{z}}^{ |\mathbf{m}_2|} = |\mathbf{z}| ^{2 \mathbf{m}_1 } \bar{ \mathbf{z}}^{ |\mathbf{m}_2|-\mathbf{m}_1}=  |\mathbf{z}| ^{2 \mathbf{m}_1 }\mathbf{z}^{\mathbf{m}_1+\mathbf{m}_2}$,
	which again is the desired formula.
\end{proof}

\begin{remark}
	Each component of $|\mathbf{m}_1|+|\mathbf{m}_2|-|\mathbf{m}_1+\mathbf{m}_2|$ is a nonnegative and even integer.
\end{remark}

\begin{proof}[Proof of Proposition \ref{prop:rp}]
	Recall $\boldsymbol{\phi}=(\phi_1,\cdots, \phi_N) \in \(  \Sigma^\infty \)^N$ are the eigenvectors of $H$ given in Assumption \ref{ass:linearInd}.
	We look for an approximate solution of \eqref{nls} of   form $u=\phi (\mathbf{z}(t))$ for  appropriate
	\begin{align}\label{eq:approx_sol}
	\phi (\mathbf{z}):=\mathbf{z}\cdot \boldsymbol{\phi} + \sum_{\mathbf{m}\in \mathbf{NR}_1}\mathbf{z}^{\mathbf{m}}\psi_{\mathbf{m}}(|\mathbf{z}|^2),
	\end{align}
	with real valued $\psi_{\mathbf{m}}$ and   orthogonality conditions $\<\psi_{\mathbf{e}_j}, \phi_j\>=0$ for all $j\in \{1,\cdots,N\}$.
	We set
	\begin{align}\label{eq:def_phi_j}
	\widetilde \phi_{\mathbf{m}}(|\mathbf{z}|^2):=\begin{cases} \phi_j +\psi_{\mathbf{e}_j}(|\mathbf{z}|^2) &  \text{  if }  \mathbf{m}=\mathbf{e}_j ,\\ \psi_{\mathbf{m}}(|\mathbf{z}|^2) &  \text{  if }  \mathbf{m}\in  \mathbf{NR}_1\setminus \mathbf{NR}_0.
	\end{cases}
	\end{align}
	
	\begin{remark} We will show that
		$\widetilde \phi_{\mathbf{m}}(|\mathbf{z}|^2)$ for $\mathbf{z}=0$ are equal to the $\widetilde \phi_{\mathbf{m}}(0)$ given in \eqref{eq:indefphigroot} and \eqref{eq:indefphi}.
	\end{remark}
	Assuming $z_j(t)=e^{-\im \varpi_j(|\mathbf{z}|^2) t}z_j$,  with $\varpi_j$ to be determined, from  $\frac{d}{dt}|z_j(t)|^2=0$
	we have
	\begin{align}\label{eq:rptd}
	\im \partial_t \phi (\mathbf{z}) = \sum_{\mathbf{m} \in \mathbf{NR}_1}\mathbf{z}^{\mathbf{m}} \(\boldsymbol{\varpi}\cdot\mathbf{m}\) \widetilde \phi_{\mathbf{m}}.
	\end{align}
	Next, we have
	\begin{align}\label{eq:rpH}
	H\phi (\mathbf{z}) = \sum_{\mathbf{m} \in \mathbf{NR}_1}\mathbf{z}^{\mathbf{m}} H \widetilde \phi_{\mathbf{m}}.
	\end{align}
	We need to Taylor expand the nonlinearity $g$ till the remainder becomes sufficiently small.
	We  will expand now  $g(|\phi (\mathbf{z})|^2)\phi (\mathbf{z})=\sum_{\mathbf{m} \in \mathbf{NR}_1}\mathbf{z}^{\mathbf{m}}g_{\mathbf{m}} + \widetilde R$ with $\| \widetilde R \|_{\Sigma^s}\lesssim_s \|\mathbf{z}\|^2\sum_{\mathbf{R}_{\mathrm{min}}}|\mathbf{z}^{\mathbf{m}}|$.
	We start with
	\begin{align*}
	|\phi (\mathbf{z})|^2 &= \(\sum_{\mathbf{m}_1\in \mathbf{NR}_1}\mathbf{z}^{\mathbf{m}_1}\widetilde \phi_{\mathbf{m}_1}\)  \(\sum_{\mathbf{m}_2\in \mathbf{NR}_1}\mathbf{z}^{-\mathbf{m}_2}\widetilde \phi_{\mathbf{m}_2}\) \\&=\sum_{\mathbf{m}\in \mathbf{NR}_1} |\mathbf{z}|^{2|\mathbf{m}|} \widetilde \phi_{\mathbf{m}}^2 +\sum_{\substack{\mathbf{m}_1,\mathbf{m}_2\in  \mathbf{NR}_1 \\ \mathbf{m}_1\neq \mathbf{m}_2}}  \mathbf{z}^{\mathbf{m}_1} \mathbf{z}^{-\mathbf{m}_2} \widetilde \phi_{\mathbf{m}_1}\widetilde \phi_{\mathbf{m}_2}.
	\end{align*}
	
	\begin{claim}\label{claim:rem1}
		Assume $\|\widetilde \phi_{\mathbf{m}}\|_{\Sigma^s}\lesssim_s 1$ for all $\mathbf{m}\in \mathbf{NR}_1$.
		Then, there exists $M>0$ s.t.\ for all $\mathbf{z} \in \C^N$ with $\|\mathbf{z}\|\leq 1$,
		\begin{align}
		\left \|\(\sum_{\mathbf{m}_1\neq \mathbf{m}_2} \mathbf{z}^{\mathbf{m}_1} \mathbf{z}^{-\mathbf{m}_2} \widetilde \phi_{\mathbf{m}_1}\widetilde \phi_{\mathbf{m}_2}\)^{M+1}\(\sum_{\mathbf{m}_3\in \mathbf{NR}_1}\mathbf{z}^{\mathbf{m}_3} \widetilde \phi_{\mathbf{m}_3}\)\right \|_{\Sigma^s} \lesssim_s \|\mathbf{z}\|^2\sum_{\mathbf{m}\in \mathbf{R}_{\mathrm{min}}}|\mathbf{z}^{\mathbf{m}}|. \label{eq:claim:rem1}
		\end{align}
	\end{claim}
	
	\begin{proof} An $M\in \N$ such that $\omega_1+ M\min_{1\leq j\leq N-1} \(\omega_{j+1}- \omega_j\)>0$ will work. To begin, we remark that   for  $\|\mathbf{z}\|\leq 1$ we have $|\mathbf{z}^{\mathbf{m}_1} \mathbf{z}^{-\mathbf{m}_2}|\leq \|\mathbf{z}\| ^2$ for $\mathbf{m}_1\neq \mathbf{m}_2$.
		Indeed, by Lemma \ref{lem:zmmult} this can only fail if $|\mathbf{m}_1|+|\mathbf{m}_2|-|\mathbf{m}_1-\mathbf{m}_2| =0$. This implies $ m_{1j}m_{2j}\ge 0$ for all $j=1,...,N$.
		Furthermore, if the inequality  fails, we can reduce to the case  $ {|\mathbf{m}_1|-|\mathbf{m}_2|}= \mathbf{e}_{j_0}$ for an index $j_0$. So $m_{1j}=m_{2j}$ for all $j\neq j_0$, and $m_{1j_0}=m_{2j_0}\pm 1$. This is incompatible
		with $\sum \mathbf{m}_1 =\sum \mathbf{m}_2 =1$.
		
		\noindent With the above remark, we can take one of the factors of the $M+1$--th power in \eqref{eq:claim:rem1} bounding it with
$\|\mathbf{z}\|^2$, concluding that  to prove \eqref{eq:claim:rem1}   it suffices to show that   for   $\mathbf{m}_{1j}, \mathbf{m}_{2j},\mathbf{m}_3\in \mathbf{NR}_1$ with $\mathbf{m}_{1j}\neq \mathbf{m}_{2j}$, there exists $\mathbf{m}\in \mathbf{R}_{\min}$ s.t.
		\begin{align}\label{eq:res}|\mathbf{z}^{\mathbf{m}_3}|\prod_{j=1}^{M } |\mathbf{z}^{\mathbf{m}_{1j}} \mathbf{z}^ {-\mathbf{m}_{2j}} |\le
		|\mathbf{z}^{\mathbf{m}_3}|\prod_{j=1}^{M } |\mathbf{z}^{\mathbf{m}_{1j}-\mathbf{m}_{2j}} |\le  |\mathbf{z}^{\mathbf{m}}|  \text{  when $\|\mathbf{z}\|\leq 1$},
		\end{align}
		where the first inequality follows from Lemma \ref{lem:zmmult}. Noticing that complex conjugation does not change absolute value, we conclude that each factor $|\mathbf{z}^{\mathbf{m}_{1j}-\mathbf{m}_{2j}}|$ has at least one factor   $|z_{a_j } \overline{z}_{b_j  }|$ with  $a_j >b_j $. There is a nonzero component $  m _{3k}\neq 0$  of $\mathbf{m}_3$. Set
		\begin{equation*}
		\mathbf{n}:=\mathbf{e}_{k}+\sum _{j=1}^{M }\(  \mathbf{e}_{a_j}- \mathbf{e}_{b_j}  \) .
		\end{equation*}
		Obviously $\sum \mathbf{n}=1$. Moreover,   $\mathbf{n}\in \mathbf{R}$,  since, by our choice of $M$,
		\begin{align*}
		\boldsymbol{\omega} \cdot  \mathbf{n}=\omega_{k}+\sum _{j=1}^{M }\(  \omega_{a_j}- \omega _{b_j}  \)  \ge \omega_1+ M \min_{1\leq j\leq N-1} \(\omega_{j+1}- \omega_j\)>0.
		\end{align*}
		But for any  $\mathbf{n}\in \mathbf{R}$ there exists an $\mathbf{m}\in \mathbf R_{\mathrm{min}}$ s.t.\ $|\mathbf{m}|\preceq |\mathbf{n}|$. Obviously, all the factors of   the l.h.s. of \eqref{eq:res}
		which we ignored are $\leq 1$. This proves \eqref{eq:res} and completes the proof of Claim  \ref{claim:rem1}.
	\end{proof}
	
	We consider a Taylor expansion
	\begin{align}\nonumber
	&g(|\phi (\mathbf{z})|^2)\phi (\mathbf{z}) \\& =  \(\sum_{m=0}^M \frac{1}{m!}g^{(m)}\(\sum_{\mathbf{m} \in \mathbf{NR}_1} |\mathbf{z}|^{2|\mathbf{m}|} \widetilde \phi_{\mathbf{m}}^2\) \(\sum_{\mathbf{m}_1\neq \mathbf{m}_2} |\mathbf{z}|^{|\mathbf{m}_1|+|\mathbf{m}_2|-|\mathbf{m}_1-\mathbf{m}_2|} \mathbf{z}^{\mathbf{m}_1-\mathbf{m}_2} \widetilde \phi_{\mathbf{m}_1}\widetilde \phi_{\mathbf{m}_2}\)^{m}\) \nonumber\\&\quad \times\(\sum_{\mathbf{m}_3\in \mathbf{NR}_1}\mathbf{z}^{\mathbf{m}_3} \widetilde \phi_{\mathbf{m}_3}\)+\widetilde{\mathcal R},\label{eq:rp1}
	\end{align}
	where $\widetilde{\mathcal R} = \mathcal O\(\|\mathbf{z}\|^2\sum_{\mathbf{m}\in \mathbf{R}_{\mathrm{min}}}|\mathbf{z}^{\mathbf{m}}|\)$,   by Claim \ref{claim:rem1}, and so can be absorbed in the   $\mathcal{R}(\mathbf{z}(t))$
	in \eqref{eq:nlsforce}.
	Thus, we only have to consider the contribution of the summation.
	For $0\leq m\leq M$, we have
	\begin{align*}
	&\(\sum_{\mathbf{m}_1\neq \mathbf{m}_2} |\mathbf{z}|^{|\mathbf{m}_1|+|\mathbf{m}_2|-|\mathbf{m}_1-\mathbf{m}_2|} \mathbf{z}^{\mathbf{m}_1-\mathbf{m}_2} \widetilde \phi_{\mathbf{m}_1}\widetilde \phi_{\mathbf{m}_2}\)^{m}\(\sum_{\mathbf{m}_3\in \mathbf{NR}_1}\mathbf{z}^{\mathbf{m}_3} \widetilde \phi_{\mathbf{m}_3}\)\\&
	=\sum_{\substack{\mathbf{m}_{1j}\neq \mathbf{m}_{2j}\\ \mathbf{m}_3}}|\mathbf{z}|^{\sum_{j=1}^m\(|\mathbf{m}_{1j}|+|\mathbf{m}_{2j}|\)+|\mathbf{m}_3|-|\sum_{j=1}^m (\mathbf{m}_{1j}-\mathbf{m}_{2j})+\mathbf{m}_3|} \(\prod_{j=1}^m \widetilde\phi_{\mathbf{m}_{1j}} \widetilde\phi_{\mathbf{m}_{2j}} \)\widetilde\phi_{\mathbf{m}_{3}}\mathbf{z}^{\sum_{j=1}^m(\mathbf{m}_{1j}-\mathbf{m}_{2j})+\mathbf{m}_3}
	\end{align*}
	Thus, if for each $\mathbf{m} \in \Z^N$, we set
	\begin{align*}
	g_\mathbf{m}&:=g_{\mathbf{m}}(|\mathbf{z}|^2,\{\psi_{\mathbf{m}}\}_{\mathbf{m}\in \mathbf{NR}_1}):=\sum_{m=0}^M \frac{1}{m!}g^{(m)}\(\sum_{\mathbf{n} \in \mathbf{NR}_1} |\mathbf{z}|^{2|\mathbf{n}|} \widetilde \phi_{\mathbf{n}}^2\) \\&\quad \times \sum_{\substack{\mathbf{m}_3,\mathbf{m}_{kj}\in \mathbf{NR}_1,\ k=1,2,\ j=1,\cdots,m\\ \sum_{j=1}^m (\mathbf{m}_{1j}-\mathbf{m}_{2j})+\mathbf{m}_3=\mathbf{m} \\ \mathbf{m}_{1j}\neq \mathbf{m}_{2j}}} |\mathbf{z}|^{\sum_{j=1}^m\(|\mathbf{m}_{1j}|+|\mathbf{m}_{2j}|\)+|\mathbf{m}_3|-|\mathbf{m}|}\(\prod_{j=1}^m \widetilde\phi_{\mathbf{m}_{1j}} \widetilde\phi_{\mathbf{m}_{2j}} \)\widetilde\phi_{\mathbf{m}_{3}},
	\end{align*}
	for $\widetilde{ \mathcal R}$   the term given in \eqref{eq:rp1}  we obtain
	\begin{align}\label{eq:rpn}
	g(|\phi (\mathbf{z})|^2)\phi (\mathbf{z})=\sum_{\mathbf{m}\in \mathbf{NR}_1}\mathbf{z}^{\mathbf{m}}g_{\mathbf{m}}+\sum_{\mathbf{m}\not\in \mathbf{NR}_1}\mathbf{z}^{\mathbf{m}}g_{\mathbf{m}}+\widetilde{ \mathcal R}.
	\end{align}
	
	\begin{remark} Notice that
		$\left . g_{\mathbf{m}}(|\mathbf{z}|^2,\{\psi_{\mathbf{m}}(|\mathbf{z}|^2)\}_{\mathbf{m} \in \mathbf{NR}_1}) \right |_{\mathbf{z}=0}$ coincides with the  $g_{\mathbf{m}}(0)$   in \eqref{eq:indefphigroot} and \eqref{eq:indefg}.
	\end{remark}
	Summing up,   we obtain  the following (where in the 2nd line we have a finite sum)
	\begin{align} \label{eq:approx_sol_1}\im \partial_t \phi (\mathbf{z}) - H \phi (\mathbf{z}) - g(|\phi (\mathbf{z})|^2)\phi (\mathbf{z})=
	&\sum _{\mathbf{m}\in \mathbf{NR}_1} \mathbf{z^m}  \( (\boldsymbol{\varpi}\cdot \mathbf{m})\widetilde \phi_{\mathbf{m}} - H\widetilde \phi_{\mathbf{m}}-g_{\mathbf{m}} \) \\&- \sum_{\mathbf{m}\not\in \mathbf{NR}_1}\mathbf{z}^{\mathbf{m}}g_{\mathbf{m}}-\widetilde{ \mathcal R}. \nonumber
	\end{align}
	Notice that, by the definition of $\mathbf{NR}_1$ and $\mathbf{R}_{\mathrm{min}}$, we have
	\begin{align*}
	\|\sum_{\mathbf{m}\not\in \mathbf{NR}_1\cup \mathbf{R}_{\mathrm{min}} }\mathbf{z}^{\mathbf{m}}g_{\mathbf{m}}\|_{\Sigma^s} \lesssim \|\mathbf{z}\|^2 \sum_{\mathbf{m}\in \mathbf{R}_{\min}} |\mathbf{z}^{\mathbf{m}}|.
	\end{align*}
	Thus, entering $G_{\mathbf{m}}$ for $\mathbf{m}\in \mathbf{R}_{\mathrm{min}}$ defined in \eqref{eq:defG}
	and
	\begin{align*}
	\mathcal R(\mathbf{z}):=\sum_{\mathbf{m}\not\in \mathbf{NR}_1\cup \mathbf{R}_{\mathrm{min}}}\mathbf{z}^{\mathbf{m}}g_{\mathbf{m}} + \sum_{\mathbf{m}\in \mathbf{R}_{\mathrm{min}}}\mathbf{z}^{\mathbf{m}}\(g_{\mathbf{m}}(|\mathbf{z}|^2)-G_{\mathbf{m}}\) + \widetilde{ \mathcal R},
	\end{align*}
	we have the estimate \eqref{est:R}. Thus
	the proof of Proposition \ref{prop:rp} follows  if the 1st summation in the r.h.s. of   \eqref{eq:approx_sol_1} cancels out, that is,  if we solve   the system
	\begin{align} \label{eq:system_recur}
	(\boldsymbol{\varpi}\cdot \mathbf{m})\widetilde \phi_{\mathbf{m}}=H\widetilde \phi_{\mathbf{m}}+g_{\mathbf{m}},\ \mathbf{m}\in \mathbf{NR}_1.
	\end{align}
	Here the unknowns are $\boldsymbol{\varpi}$ and   $ \psi_{\mathbf{m}}$, since the latter determines $\widetilde \phi_{\mathbf{m}}$ by \eqref{eq:def_phi_j}, while
	$g_{\mathbf{m}}$ are   given  functions of both the variables $|\mathbf{z}|^2$ and $\{\psi_{\mathbf{m}}\}_{\mathbf{m} \in \mathbf{NR}_1}$.
	We will later determine $\{\psi_{\mathbf{m}}\}_{\mathbf{m} \in \mathbf{NR}_1}$ as a function of $|\mathbf{z}|^2$, and so at the end $\boldsymbol{\varpi}$  and $g_{\mathbf{m}}$
	will depend only on $|\mathbf{z}|^2$.

	\noindent We first focus on \eqref{eq:system_recur} for $\mathbf{m} =\mathbf{e}_j$ splitting in the direction parallel to $\phi _j$
	and the space orthogonal to  $\phi _j$. In the direction  parallel to $\phi _j$, that is taking inner product with $\phi_j$ (and recalling assumption $\<\psi_{\mathbf{e}_j}, \phi_j\>=0$),
	we have
	\begin{align}\label{eq:system_recur0}
	\varpi_j(|\mathbf{z}|^2,\{\psi_{\mathbf{m}}\}_{\mathbf{m}\in \mathbf{NR}_1}) = \omega_j + \<g_{\mathbf{e}_j}(|\mathbf{z}|^2,\{\psi_{\mathbf{m}}\}_{\mathbf{m} \in \mathbf{NR}_1}), \phi_j\>.
	\end{align}
	This determines    $\boldsymbol{\varpi}$ as a function of $|\mathbf{z}|^2$ and $\{\psi_{\mathbf{m}}\}_{\mathbf{m} \in \mathbf{NR}_1}$.
	Later  we will determine $\{\psi_{\mathbf{m}}\}_{\mathbf{m} \in \mathbf{NR}_1}$ as a function of $|\mathbf{z}|^2$, so in the end $\boldsymbol{\varpi}$ will be a function of $|\mathbf{z}|^2$. Notice also that
	$\varpi_j(0,\{\psi_{\mathbf{m}}\}_{\mathbf{m}\in \mathbf{NR}_1}) =\omega_j$ because $g_{\mathbf{e}_j}(0,\{\psi_{\mathbf{m}}\}_{\mathbf{m} \in \mathbf{NR}_1})=0$, as can be seen from the definition of $g_\mathbf{m}$.

	\noindent Next, set
	\begin{align*}
	A_\mathbf{m}:=\begin{cases}
	\(\left.(H-\omega_j)\right|_{\{\phi_j\}^\perp}\)^{-1} & \mathbf{m}=\mathbf{e}_j\in \mathbf{NR}_0,\\ (H-\mathbf{m}\cdot \boldsymbol{\omega})^{-1} & \mathbf{m} \in \mathbf{NR}_1\setminus \mathbf{NR}_0.
	\end{cases}
	\end{align*}
	The following lemma is standard and we skip the proof.
	\begin{lemma}\label{lemma:invH}
		For all $\mathbf{m}\in \mathbf{NR}_1$ and any $s\in \R$  we have $\|A_{\mathbf{m}}\|_{\Sigma^s\to \Sigma^{s+2}}\lesssim_s 1$.
	\end{lemma}
	It is elementary that \eqref{eq:system_recur} holds if and only if both \eqref{eq:system_recur0}  and   the following  system hold:
	\begin{align}\label{eq:system_recur1}
	F_{\mathbf{m}}(|\mathbf{z}|^2,\{\psi_{\mathbf{n}}\}_{\mathbf{n}\in\mathbf{NR}_1}):=\psi_{\mathbf{m}}-A_\mathbf{m} \( (\boldsymbol{\varpi}-\boldsymbol{\omega})\cdot\mathbf{m} \psi_{\mathbf{m}} -g_\mathbf{m}\)=0,\ \mathbf{m}\in \mathbf{NR}_1.
	\end{align}
	We have $\{F_{\mathbf{m}}\}_{\mathbf{m}\in\mathbf{NR}_1}\in C^\infty(\R^N\times \Sigma^s_{\mathbf{NR}_1},\Sigma^s_{\mathbf{NR}_1})$, for
	\begin{align*}
	\Sigma^{s}_{\mathbf{NR}_1}=\Sigma^{s}_{\mathbf{NR}_1}(\R^3, \R) :=\left\{ \{\psi_{\mathbf{m}}\}_{\mathbf{m}\in\mathbf{NR}_1}\in (\Sigma^s(\R^3, \R))^{\sharp\mathbf{NR}_1}\ |\ \<\psi_{\mathbf{e}_j},\phi_j\>=0,\ j=1,\cdots, N\right\}.
	\end{align*}

	Since, for $D_{\{\psi_{\mathbf{m}}\}_{\mathbf{m}\in \mathbf{NR}_1}}F$   the Fr\'echet derivative of $F$ w.r.t.\ the $\{\psi_{\mathbf{m}}\}_{\mathbf{m}\in\mathbf{NR}_1}$,
	\begin{align*}
	\{F_{\mathbf{m}}(0,0)\}_{\mathbf{m}\in \mathbf {NR}_1}=0 \text{  and } D_{\{\psi_{\mathbf{m}}\}_{\mathbf{m}\in \mathbf{NR}_1}}F(0,0)=\mathrm{Id}_{\Sigma^s_{\mathbf{NR}_1}},
	\end{align*}by implicit function theorem  there exist  $\delta_s>0$  and $\{\psi_{\mathbf{m}}(\cdot)\}_{\mathbf{m}\in \mathbf {NR}_1}\in C^\infty(B_{\R^N}(0,\delta_s),\Sigma^s_{\mathbf{NR}_1})$ s.t.\
	\begin{align*}
	F_{\mathbf{m}}(|\mathbf{z}|^2, \{\psi_{\mathbf{n}}(|\mathbf{z}|^2)\}_{\mathbf{n}\in \mathbf {NR}_1})=0,\ \mathbf{m}\in \mathbf {NR}_1.
	\end{align*}
	Setting  $\boldsymbol{\varpi}(|\mathbf{z}|^2):=\boldsymbol{\varpi}(|\mathbf{z}|^2,\{\psi_{\mathbf{n}}(|\mathbf{z}|^2)\}_{\mathbf{n}\in \mathbf {NR}_1} )$,  $g_{\mathbf{m}}(|\mathbf{z}|^2):=g_{\mathbf{m}}(|\mathbf{z}|^2,\{\psi_{\mathbf{n}}(|\mathbf{z}|^2)\}_{\mathbf{n}\in \mathbf {NR}_1})$,
	and $u(t)=\phi(\mathbf{z}(t))$ with $z_j(t)=e^{-\varpi_j(|\mathbf{z}|^2)t}z_j$ and  $\boldsymbol{\phi}  $  defined in \eqref{eq:approx_sol}, we obtain the conclusions of Proposition\ref{prop:rp}.
\end{proof}

\begin{remark}\label{re:zero}
	From \eqref{eq:system_recur1} we have $\psi_{\mathbf{e}_j} (0) =0$, since $g_{\mathbf{e}_j} (0) =0$  and $\left . \varpi_j(\mathbf{z},\{\psi_{\mathbf{m}}\}_{\mathbf{m}\in \mathbf{NR}_1}) \right | _{\mathbf{z}=0} =\omega_j$, as we remarked under \eqref{eq:system_recur0}.
\end{remark}

\section{Darboux theorem and proof of Proposition \ref{prop:Darcor}}\label{sec:Dar1}




In   this  section, we will always assume $B\subset_{\mathrm{dense}} H$  (i.e.\ $B$ is a dense subset of $H$) where $B$ is a reflexive Banach space and $H$ is a Hilbert space.
We further always identify $H^*$ with $H$ by the isometric isomorphism $H\ni u\mapsto \<u,\cdot\> \in H^*$, where $\<\cdot,\cdot\>$ is the inner product of $H$.
We will also denote the coupling between $B^*$ and $B$ by $\<f,u\>$.

When we have $B\subset_{\mathrm{dense}} H \subset_{\mathrm{dense}} B^*$, we think $B$ as a ``regular" subspace of $H$ and $B^*$.
We introduce several notation.

\begin{definition}
	Let $U\subset B^*$.
	Let $\varphi$ be a $C^\infty$-diffeomorphism from  $U\subset B^*$ to $\varphi(U) \subset B^*$.
	We call $\varphi$ a ($B$-)almost identity if $\varphi(u)-u \in C^\infty(U,B)$.
\end{definition}

\begin{definition}
	Let $U\subset B^*$.
	We define $B$-regularizing vector fields $\mathfrak{X}_\sharp(U)$ and  regularizing 1--forms $\Omega^1_\sharp(U)$ by
	\begin{align*}
	\fxr(U):=C^\infty(U,B) \subset \fx(U):=C^\infty(U,B^*)\text{ and }\Or^1(U):=\Omega^1(U).
	\end{align*}
	Here, for a Banach space $B_1$ and an open subset $U_1\subset B_1$, the space of $k$-forms are given by $\Omega^k(U_1,\mathcal{L}^k_a(B_1,\R))$ where $\mathcal L^k_a(B_1,\R)$ is the Banach space of anti-symmetric $k$-linear operators.
\end{definition}

\begin{remark}
	$B$-regularizing 1--forms are mere 1--forms on $B^*$.
	However, if we think ``standard" 1--forms as differential forms defined on $H$, $B$-regularizing 1--forms are more regular than ``standard" 1--forms because they make sense with more ``rough" vectors which are in $B^*$ and not in $H$.
	We further remark
	\begin{align*}
	\Or^1(U)=C^\infty(U,\mathcal L(B^*,\R))\simeq C^\infty(U,B)=\fxr(U).
	\end{align*}
\end{remark}

\begin{definition}[Symplectic forms]
	Let $U\subset H$.
	We say that $\Omega \in \Omega^2(U)$ is a    symplectic form on $U$ if $d \Omega=0$ and if for each $u\in U$ the following map is an isomorphism:
	\begin{align}\label{def:symplector}
	H\ni v\mapsto \Omega(u)(v,\cdot)\in  H^*\simeq H
	\end{align}
	Following  \cite{BGR15Survey},  we call the map  in \eqref{def:symplector} symplector, denoting it by $J(u)$.
	The symplector satisfies $J\in C^\infty(U,\mathcal L(H))$.
	If there exists an open set $V\subset B^*$ s.t.\ $U\subset V$ and we can extend $J$ and $J ^{-1 } $ on $V$ so that both  are in  $  C^\infty(V,\mathcal L(B))$, we say that $\Omega$ is a $B$-compatible  symplectic form.
\end{definition}

\begin{remark}
	Our symplectic form is the strong symplectic form of \cite{AMRBook}.
\end{remark}

Let $\Omega$ be a symplectic form and $J$ be the associated symplector.
Then, we have
\begin{align}\label{eq:symplectorrep}
\Omega(u)(X,Y)=\<J(u)X,Y\>,\ X,Y\in H.
\end{align}
Moreover, if $\Omega$ is a $B$-compatible symplectic form, then for $X\in B$ and $Y\in B^*$, we can define $\Omega(u)(X,Y)$ by \eqref{eq:symplectorrep}.
Of course we can also define $\Omega(u)(Y,X)$ by $\Omega(u)(Y,X):=-\Omega(u)(X,Y)$.

The symplector corresponding to the symplectic form $\Omega_0$ given in \eqref{def:ssymp} is $J(u)=\im$.
Obviously, this symplectic form is $B$ compatible for any Banach space $B\subset_{\mathrm{dense}} H$ satisfying the property $f\in B\Rightarrow \im f\in B$.
In particular, if $H=L^2(\R^3,\C)$ and $B=H^s_{\gamma}(\R^3,\C)$, $\Omega_0$ is a $H^s_{\gamma}(\R^3,\C)$-compatible symplectic form.

We next consider a small perturbation of a $B$-compatible symplectic form.

\begin{lemma}\label{lem:symppert}
	Let $U_1\subset H$ and $U_2\subset B^*$ with $U_1\subset U_2$, $\Omega$ be a $B$-compatible symplectic form and $F\in \Omega_\sharp^1(U_2)$.
	Let $u_0\in U_1$ and assume $F(u_0)=0$ and $DF(u_0)=0$.
	Then, there exists an open set $V\subset U_1$ in $H$ s.t.\ $\Omega+dF$ is an $B$-compatible symplectic form on $V$, where $dF \in\Omega^2(U)$ is the exterior derivative of $F$.
\end{lemma}

\begin{remark}
	Since $F\in  \Or^1(U)=C^\infty(U,\mathcal L(B^*,\R))$, we have $DF \in C^\infty(U,\mathcal L^2(B^*,\R))$.
	If $DF(u_0)=0$, then from the definition of exterior derivative, we have $dF(u_0)=0$ too.
\end{remark}

\begin{proof}
	We identify $\mathcal L^2(B^*,\R)$ with $\mathcal L(B^*,\mathcal L(B^*,\R))=\mathcal L(B^*,B^{**})\simeq \mathcal L(B^*,B)$.
	In this case, we can write $DF(u)(X,Y)=\<DF(u)X,Y\>$ for $X,Y\in B^*$.
	Therefore, we have
	\begin{align*}
	dF(u)(X,Y)=\<DF(u)X,Y\>-\<DF(u)Y,X\>=\<\(DF(u)-(DF(u))^*\)X,Y\>,
	\end{align*}
	where $(DF(u))^*\in \mathcal L(B^*,B)$ is the adjoint of $DF(u)\in \mathcal L(B^*,B)$.
	Thus, for $J \in C^\infty(U_1,\mathcal L(H))$  the symplector of $\Omega$, we have
	\begin{align}\label{eq:symppert}
	\Omega(u)(X,Y)+dF(u)(X,Y)=\<\(J(u)+DF(u)-(DF(u))^*\)X,Y\>.
	\end{align}
	Since $DF(u_0)=0$, there exists an open neighborhood $V$ of $u_0$ in $B^*$ s.t.\ $J(u)+DF(u)-(DF(u))^*$ is invertible for all $u\in V$.
	Hence $\Omega+dF$ is a symplectic form with symplector
$J(u)+DF(u)-(DF(u))^*$.
	Since $DF(u)-(DF(u))^* \in \mathcal L(B^*,B)$, the restriction of $DF(u)-(DF(u))^*$ to $B$ is in $\mathcal L(B)$.
	Therefore, we have the conclusion.
\end{proof}

We are now in the  position to prove an Darboux theorem with appropriate error estimates.

\begin{proposition}[Darboux theorem]\label{prop:darboux}
	Let $U_1\subset H$ and $U_2\subset B^*$ be open sets with $U_1\subset U_2$ and  let $\Omega_1$ be a $B$-compatible symplectic form and $F\in \Omega_\sharp^1(U_2)$.
	Let $u_0\in U_1$ and assume $F(u_0)=0$ and $DF(u_0)=0$. Set $\Omega_2:=\Omega_1+d F$.
	Then, there exists an open neighborhood $V\subset U_2$ of $u_0$ in $B^*$ and a map  $\varphi  \in C^\infty(V, B^*)$   s.t.\
	$
	\varphi^* \Omega_2=\Omega_1,
	$
	and
	\begin{align}\label{eq:Darerr}
	\forall u\in V,\ \| \varphi(u) - u \|_{B}\lesssim \|F(u)\|_{B}.
	\end{align}
\end{proposition}

\begin{proof}
	We first set $$\Omega_{s+1}:= \Omega_1 +s(\Omega_2-\Omega_1)=\Omega_1+s dF,$$ and look for a vector field $\mathcal X_{s+1}$ that satisfies $i_{\mathcal X_{s+1}}\Omega_{s+1} :=\Omega_{s+1}(\mathcal X_{s+1},\cdot)= -F$.
	\begin{claim}\label{claim:darb1}
		There exists an open neighborhood $V_1$ of $u_0$ in $B^*$ s.t.\
		there exists
		\begin{align}\label{eq:dar1}
		\mathcal X_{\cdot +1} \in C^\infty((-2,2) \times V_1,B)\ \text{satisfying } i_{\mathcal X_{s+1}} \Omega_{s+1} = - F
		\end{align}
		and such that  there exists a $C_1>0$ s.t.
		\begin{align}\label{eq:darest1}
		\sup_{s\in (-2,2)}\|\mathcal X_{s+1}(u)\|_{B}\le C_1  \|F(u)\|_{B} \text{  for all $u\in V_1$.}
		\end{align}
	\end{claim}
	
	\noindent  \textit{Proof of Claim \ref{claim:darb1}.}
	Since $F\in \Or^1(U_2)=C^\infty(U_2,\mathcal L(B^*,\R))\simeq C^\infty(U_2,B)$, we can express $F(u)X=\<F(u),X\>$ for any $X\in B^*$.
	Therefore, using also \eqref{eq:symppert}, \eqref{eq:dar1} can be expressed as
	\begin{align*}
	(J_1(u)+s\(DF(u)-(DF(u))^*\))\mathcal X_{s+1} = -F(u),
	\end{align*}
	where $J_1 \in C^\infty(U,\mathcal L(H))$ is the symplector of $\Omega_1$.
	Since $\Omega_1$ is a $B$-compatible symplectic form, we can extend $J_1$ to $J_1\in C^\infty(V',\mathcal L(B))$ for some $U\subset V'$ with $V'\subset B^*$.
	
	Take now $\delta_1>0$ sufficiently small, so that $\overline{\mathcal{B}_{B^*}(u_0,\delta_1)}\subset V'$ and
	\begin{align*}
	\sup_{u\in {\mathcal{B}_{B^*}(u_0,\delta_1)}}\|DF(u )\|_{\mathcal L(B^*,B)}\leq (8\sup_{u \in {\mathcal{B}_{B^*}(u_0,\delta_1)}} \|J_1(u)^{-1}\|_{\mathcal L(B)})^{-1}.
	\end{align*}
	Then, we have
	\begin{align*}
	\forall u\in  {\mathcal{B}_{B^*}(u_0,\delta_1)},\ \|J_1(u)^{-1}\(DF(u)-(DF(u))^*\) \|_{\mathcal L(B^*,B)}\leq \frac14
	\end{align*}
	Thus, by Neumann series we have
	\begin{align*}
	\mathcal X_{s+1}(u)=-\sum_{n=0}^\infty \( sJ_1(u)^{-1}\(DF(u)-(DF(u))^*\)\)^nJ_1(u)^{-1}F(u),
	\end{align*}
	where the r.h.s.\ absolutely converges uniformly for $s\in (-2,2)$.
	Therefore, setting $V_1=\mathcal{B}_{B^*}(u_0,\delta_1)$, we have the conclusion.
	\qed
	
	\begin{claim}\label{claim:darb2}
		There exit  an open neighborhood $V_2\subset V_1 $ of $u_0$   in  $B^* $ and a   map $\widetilde\varphi_{\cdot} \in C^\infty((-2,2)\times V_2,B )$   s.t.\
		\begin{align}\label{eq:darode}
		\frac{d}{ds}\widetilde\varphi_{s }(u)=\mathcal X_{s+1}(u+\widetilde\varphi_{s }(u)),\quad \widetilde\varphi_0(u)=0
		\end{align}
		and
		\begin{align}\label{eq:darodeest}
		\sup_{s\in[0,1]}\|\widetilde\varphi_s(u)\|_B\le 2C_1 \|F(u)\|_B  \text{  for all $u\in V_2$.}
		\end{align}
	\end{claim}
	
	\noindent \textit{Proof of Claim \ref{claim:darb2}.}
	The existence of $\widetilde\varphi$ satisfying \eqref{eq:darode} is standard so we concentrate on the estimate \eqref{eq:darodeest}.
	First, by the assumption $F(u_0)=0$ and $DF(u_0)=0$, there exists $\delta_2\in (0,\delta_1]$ s.t.\ for $u\in \overline{\mathcal{B}_{B^*}(u_0,\delta_2)}$, we have $\|F(u)\|_B \leq \frac{1}{4C_1} \|u-u_0\|_{B^*}$, where $C_1>0$ is the constant \eqref{eq:darest1}.
	We define $s^*(u)\in [0,1]$ by
	\begin{align*}
	s^*(u):=\min(\inf\{s\in (0,2)\ |\ \varphi_s(u)\not\in \overline{\mathcal{B}_{B^*}(u_0,\delta_2/2)}\},\ 1).
	\end{align*}
	Then, since $\widetilde \varphi_s$ is continuous,
	we have $s^*(u)>0$ for $u \in \mathcal{B}_{B^*}(u_0,\delta_2/2)$.
	Furthermore,
	\begin{align*}
	\sup_{s\in[0,s^*(u)]} \|\widetilde \varphi_s(u)\|_{B^*} &\leq \sup_{s\in [0,s^*(u)]}\int_0^s \|\mathcal X_{s'+1}(u+\widetilde \varphi_{s'}(u))\|_{B^*} \,ds' \\&\leq C_1\sup_{s\in [0,1]} \|F(u+\widetilde \varphi_s(u))\|_B \leq \frac{1}{4}\(\|u-u_0\|_{B^*}+\sup_{s\in [0,1]} \|\widetilde \varphi_s(u)\|_{B^*}\).
	\end{align*}
	Thus, we  conclude that $s^*(u)=1$   from
	\begin{align}\label{eq:dar:estphi}
	\sup_{s\in[0,s^*(u)]}\|\widetilde\varphi_s(u)\|_{B^*}\leq \frac13 \|u-u_0\|_{B^*}\leq \frac16 \delta_2<\frac12 \delta_2.
	\end{align}
	Since $B\subset H\subset B^*$, there exists $C_2\geq 1$ s.t.\ for all $u\in B$, $\|u\|_{B^*}\leq C_2 \|u\|_B$.
	Now, take $\delta_3\in (0,\delta_2]$ s.t.\ if $u\in \overline{\mathcal{B}_{B^*}(u_0,\delta_3)}$, then $\|DF(u)\|_{\mathcal L(B^*,B)}\leq (2C_1C_2)^{-1}$.
	Then, for all $u\in \overline{\mathcal{B}_{B^*}(u_0,\delta_3/2)}$ and all $s,s_1\in[0,1]$, we have $u+s_1\widetilde \varphi_s(u)\in \overline{\mathcal{B}_{B^*}(u_0,\delta_3)}$ by \eqref{eq:dar:estphi}.
	Therefore, by Taylor expansion and \eqref{eq:dar1},
	\begin{align*}
	\sup_{s\in[0,1]}\|\widetilde \varphi_s(u)\|_{B}&\leq C_1\|F(u)\|_B+C_1\sup_{s_1\in [0,1]} \|DF(u+s_1\widetilde\varphi_s(u))\|_{\mathcal L(B^*,B)}\|\widetilde \varphi(u)\|_{B^*}\\&
	\leq C_1\|F(u)\| _{B}+\frac12 \sup_{s\in [0,1]}\|\widetilde\varphi_s(u)\|_B.
	\end{align*}
	This completes the proof of  Claim \ref{claim:darb2}.
	\qed

	We set $\varphi_s(u):=u+\widetilde\varphi _s(u)$.
	Then, by Cartan's formula, see (7.4.6) \cite{AMRBook},    we have
	\begin{align}\label{eq:Moser}
	\frac{d}{ds} \varphi_s^* \Omega_{s+1} =\varphi_s^*\(\mathcal L_{\mathcal X_{s+1}}\Omega_{s+1} + dF\)=\varphi_s^*\(\(d i_{\mathcal X_{s+1}}+i_{\mathcal X_{s+1}}d\)\Omega_{s+1} + dF\)=0.
	\end{align}
	Therefore, since $\varphi_0=\mathrm{id}$, we have
	\begin{align*}
	\Omega_2=\varphi_1^* \Omega_1.
	\end{align*}
	Setting $\varphi:=\varphi_1$, we have the conclusion.
\end{proof}

     We show now that  Proposition \ref{prop:Darcor}   follows from Proposition \ref{prop:darboux}.
\begin{proof}[Proof of Proposition \ref{prop:Darcor}]
	Let $B=\Sigma^s$ and $H=L^2$.
	Set $F:=2^{-1} \Omega_0 ( D_\mathbf{z}\phi(\mathbf{z})D\mathbf{z},\eta )$. Then, by   $F \in \Omega ^1 _\sharp (\mathcal{B}_{\Sigma^{-s}} (0,\delta_s))$ and by the identification $\Omega ^1 _\sharp (\mathcal{B}_{\Sigma^{-s}} (0,\delta_s))\simeq C^\infty(\mathcal{B}_{\Sigma^{-s}} (0,\delta_s),\Sigma^{ s})$, we have $F(u)=-\im 2^{-1}(D_\mathbf{z}\phi(\mathbf{z})D\mathbf{z} )^*\im \eta $.  Notice  the cancelation
	\begin{align*}
	(D_\mathbf{z}\phi(\mathbf{z})D\mathbf{z} )^*\im \eta =  (D_\mathbf{z}\phi(\mathbf{z})D\mathbf{z}  - D_\mathbf{z}(\mathbf{z}\boldsymbol{\phi})D\mathbf{z} )^*\im \eta .
	\end{align*}
	So from  \eqref{eq:approx_sol} and by $ \|  (D_\mathbf{z} (\mathbf{z}^{\mathbf{m}}\psi_{\mathbf{m}}(|\mathbf{z}|^2))D\mathbf{z})^*  \| _{ \mathcal{L} (\Sigma^{ -s}, \Sigma^{ s} )}  =  \|   D (\mathbf{z}^{\mathbf{m}}\psi_{\mathbf{m}}(|\mathbf{z}|^2)) D\mathbf{z}  \| _{ \mathcal{L} (\Sigma^{ -s}, \Sigma^{ s} )} $, we have
	\begin{align*}
	\| F (u)\| _{\Sigma^{ s}} \le \sum_{\mathbf{m}\in \mathbf{NR}_1}    \|  (D_\mathbf{z} (\mathbf{z}^{\mathbf{m}}\psi_{\mathbf{m}}(|\mathbf{z}|^2))D\mathbf{z})^* \im \eta \| _{\Sigma^{ s}} \le  \sum_{\mathbf{m}\in \mathbf{NR}_1}
	\|   D (\mathbf{z}^{\mathbf{m}}\psi_{\mathbf{m}}(|\mathbf{z}|^2))  D\mathbf{z} \| _{ \mathcal{L} (\Sigma^{ -s}, \Sigma^{ s} )}  \|\eta \|_{\Sigma^{-s}} .
	\end{align*}
	Next we use the fact that, for   $\mathbf{m}\in \mathbf{NR}_1$, we have
	\begin{align*} &
	\|   D_\mathbf{z} (\mathbf{z}^{\mathbf{m}}\psi_{\mathbf{m}}(|\mathbf{z}|^2)) D\mathbf{z}  \| _{ \mathcal{L} (\Sigma^{ -s}, \Sigma^{ s} )} \le   \|   D_{\mathbf{z}} (\mathbf{z}^{\mathbf{m}}\psi_{\mathbf{m}}(|\mathbf{z}|^2))   \| _{   \Sigma^{ s}  }  \| D\mathbf{z} \| _{\mathcal{L} (\Sigma^{ -s}, \C ^N )}   \le C_s  \|\mathbf{z}\|^2,
	\end{align*}
	where  for $\mathbf{m}\in \mathbf{NR}_0$,  $\|   D_{\mathbf{z}} (\mathbf{z}^{\mathbf{m}}\psi_{\mathbf{m}}(|\mathbf{z}|^2))   \| _{   \Sigma^{ s}  } \le C   \|\mathbf{z}\|^2 $ follows from Remark \ref{re:zero}, and for $\mathbf{m}\in \mathbf{NR}_1\backslash  \mathbf{NR}_0  $ it follows from $|\mathbf{z^m}|\le \|\mathbf{z}\|^3$, since $\mathbf{z^m}$ has an odd number of factors.

	\noindent Summing up,   we have proved
	\begin{align}\label{eq:Darcor2}
	\| F (u)\| _{\Sigma^{ s}} \le C_s  \|\mathbf{z}\|^2 \|\eta \|_{\Sigma^{-s}} .
	\end{align}
	Then the statement of Proposition \ref{prop:Darcor} is a consequence  of $\Omega _0=\Omega _1 +dF$
	and
	of   Proposition \ref{prop:darboux}.
\end{proof}

\appendix

\section{Proofs of Lemma \ref{lem:RminNR1isfinite}    and of Proposition \ref{lem:generic g}}
\label{append:1}

\begin{proof}[Proof of Lemma \ref{lem:RminNR1isfinite}]
For $j,k\in \{1,\cdots,N\}$, $j< k$, set $n_{jk}$ to be the smallest integer satisfying $n_{jk}( \omega_k-\omega_j)+\omega_k>0$.
Then, for $\mathbf{m}^{(jk)}=(m_1^{(jk)},\cdots,m_N^{(jk)})$ defined by
\begin{align}\label{eq:defmjk}
m_j^{(jk)}= -n_{jk},\ m_k^{(jk)}=n_{jk}+1\text{ and }m_l^{(jk)}=0\  (l\neq j,k),
\end{align}
we have $\mathbf{m}^{(jk)} \in \mathbf{R}_{\mathrm{min}}$.
Suppose $\mathbf{R}_{\mathrm{min}}$ is an infinite set.
Then, there exists $j\in \{1,\cdots,N\}$ and $\{\mathbf{m}_k\}_{k=1}^\infty\subset \mathbf{R}_{\mathrm{min}}$ s.t.\ $|m_{kj}|  \xrightarrow {k\to \infty }  \infty$.
If there exists $M>0$ s.t.\ for all  $l\neq j$, $|m_{kj}|\leq M$, then $\mathbf{m}_k$ cannot satisfy $\sum \mathbf{m}_k=1$.
Therefore, if necessary taking a subsequence, there exists $l\neq j$ s.t.\ $|m_{kl}|  \xrightarrow {k\to \infty } \infty$.
However, for $k$ sufficiently large, we have $|\mathbf{m}^{(jl)}|\prec |\mathbf{m}_k|$ with $\mathbf{m}^{(jl)} \in \mathbf{R}_{\mathrm{min}}$   defined by \eqref{eq:defmjk}. This, by the definition of $\mathbf{R}_{\mathrm{min}}$ in \eqref{eq:defRmin}, implies $\mathbf{m}_k\not \in \mathbf{R}_{\mathrm{min}}$,
contradicting the hypothesis $\mathbf{m}_k  \in \mathbf{R}_{\mathrm{min}}$.

Let $\mathbf{m}\in \mathbf{NR}_1$.
It is elementary, by the definition of $\mathbf{NR}_1$ \eqref{eq:defnr1},   that for all $\mathbf{n} \in \mathbf{R}_{\mathrm{min}}$, either there exists $j$ s.t.\ $|n_j|>|m_j|$ or $ | \mathbf{n} | = | \mathbf{m} |$.
So, for $\mathbf{n}=\mathbf{m}^{(jk)}$  in \eqref{eq:defmjk}, we have either $|\mathbf{m}|=|\mathbf{m}^{(jk)}|$ or $|m_l|<|m_l^{(jk)}|$ for $l=j$ or $k$.
Since there are finitely many  $\mathbf{m}\in \mathbf{NR}_1$ s.t.\ $|\mathbf{m}|=|\mathbf{m}^{(jk)}|$, we can assume $|\mathbf{m}|\neq |\mathbf{m}^{(jk)}|$ for all $j<k$.
Thus, for all $j<k$, we have $|m_l|< m^{(jk)}_l$ for at least one of  $l \in \{j,k\}$.
It is easy to conclude that $|m_j|\leq \max_{1\leq k<l\leq N} \(|n_{kl}|+1\)$ for all $j$  except for at most one. However, from
$\sum \mathbf{m} =1$ it is immediate that this special $j$ must satisfy  $|m_j| \leq N \max_{1\leq k<l\leq N} \(|n_{kl}|+1\)$.
Thus,   $\mathbf{m}$ is in a fixed bounded set. Hence $\mathbf{NR}_1$ is a finite set.
\end{proof}

\textit{Proof of Proposition \ref{lem:generic g}}. The simple proof is analogous to Bambusi and Cuccagna \cite[p.1444]{BC11AJM}.
 For $\mathbf{m} \in \mathbf{R}_{\mathrm{min}}$  set $\N \ni L_{\mathbf{m}} := \frac{\| \mathbf{m} \| -1}{2}$. Then
from \eqref{eq:defAmm}--\eqref{eq:defG}, for any $\mathbf{m} \in \mathbf{R}_{\mathrm{min}}$ we have
\begin{align*} &
G_{\mathbf{m}}=  N _{\mathbf{m}} \frac{g^{\(   L_{\mathbf{m}} \)}(0)}{L_{\mathbf{m}} !}   \phi ^{\mathbf{m}}  + K _{\mathbf{m}},\end{align*}
where $N _{\mathbf{m}} \in \N$ is the number of elements of $ A(L_{\mathbf{m}},\mathbf{m})$, which in this particular case is given by
the set
\begin{align*} & A(L_{\mathbf{m}},\mathbf{m})= \left\{ \{\mathbf{e}_{\ell _j}\}_{j=1}^{\|\mathbf{m} \|}\in (\mathbf{NR}_0)^{\|\mathbf{m} \| }\ |\ \sum_{j=0} ^{ L_{\mathbf{m}}} \mathbf{e}_{\ell_{2j+1}}-\sum_{j=1}^{ L_{\mathbf{m}}} \mathbf{e}_{\ell_{2j }}= \mathbf{m} \right\} ,
\end{align*}
and where
\begin{align*} & K _{\mathbf{m}} :=\sum_{1\le m < L_{\mathbf{m}}  }   \frac{1}{m!}g^{(m)}(0)\sum_{(\mathbf{m}_1,\cdots,\mathbf{m}_{2m+1})\in A(m,\mathbf{m})}\widetilde\phi_{\mathbf{m}_1}(0)\cdots \widetilde{\phi}_{\mathbf{m}_{2m+1}}(0).
\end{align*}
So, expanding we have on the sphere $S _{\mathbf{m}}=\{  \xi : |\xi |^2 = \mathbf{m}\cdot  \boldsymbol{\omega} \}$ we obtain
\begin{align*} &  \| \widehat{G} _{\mathbf{m}}  \| ^{2}_{L^2(S _{\mathbf{m}})} =  \( N _{\mathbf{m}} \frac{g^{\(   L_{\mathbf{m}} \)}(0)}{L_{\mathbf{m}} !} \) ^2   \| \widehat{\phi ^ \mathbf{m}}   \| ^{2}_{L^2(S _{\mathbf{m}})} + 2 N _{\mathbf{m}} \frac{g^{\(   L_{\mathbf{m}} \)}(0)}{L_{\mathbf{m}} !} \< \widehat{\phi ^ \mathbf{m}} , \widehat{K _{\mathbf{m}}} \> _{L^2(S _{\mathbf{m}})} +   \| \widehat{K} _{\mathbf{m}} \| ^{2}_{L^2(S _{\mathbf{m}})} .
\end{align*}
Equating the above to 0 we obtain, in view of \eqref{eq:FGRsimplified}, a quadratic equation for $g^{\(   L_{\mathbf{m}} \)}(0)$
which expresses it in terms of $(g' (0),...., g ^{(L_{\mathbf{m}}-1)} (0))$. This proves Proposition \ref{lem:generic g}.
\qed

\section*{Acknowledgments}
C. was supported by a FIRB of the University of Trieste.
M.M. was supported by the JSPS KAKENHI Grant Number 19K03579, G19KK0066A, JP17H02851 and JP17H02853.

Department of Mathematics and Geosciences,  University
of Trieste, via Valerio  12/1  Trieste, 34127  Italy.
{\it E-mail Address}: {\tt scuccagna@units.it}

Department of Mathematics and Informatics,
Graduate School of Science,
Chiba University,
Chiba 263-8522, Japan.
{\it E-mail Address}: {\tt maeda@math.s.chiba-u.ac.jp}

\end{document}